\documentclass[a4paper,11pt]{article}
\usepackage{mystyle}

\usepackage{todonotes}
\usepackage{pgfplots,pgfplotstable}
\usepgfplotslibrary{external}
\tikzsetexternalprefix{arxivfigures/}
\usepgfplotslibrary{groupplots}
\pgfplotsset{compat=1.18}

\makeatletter
\tikzset{
  planar forest/.append style={
    execute at end picture={
      \path[use as bounding box]
        ([yshift=3pt]current bounding box.north west)
        rectangle
        ([yshift=-3pt]current bounding box.south east);
    }
  }
}
\makeatother

\usepackage{hyperref}
\usepackage{etoolbox}
\newcommand{\diam}{\odot}
\newcolumntype{C}{>{$}c<{$}}
\usepackage{planarforest}
\newcommand{\forestA}{
\tikz[planar forest ] {

\node [b] at (0.0, 0.0) {  } 
;
}}

\newcommand{\forestB}{
\tikz[planar forest ] {

\node [lc] at (0.0, 0.0) { 1 } 
;
\node [lc] at (1.0, 0.0) { 1 } 
;
}}

\newcommand{\forestC}{
\tikz[planar forest ] {

\node [lc] at (0.0, 0.0) { 1 } 
child {node [lc] at (0.0, 1.0) { 1 }  
}
;
}}

\newcommand{\forestD}{
\tikz[planar forest ] {

\node [b] at (0.0, 0.0) {  } 
child {node [b] at (0.0, 1.0) {  }  
}
;
}}

\newcommand{\forestE}{
\tikz[planar forest ] {

\node [b] at (0.0, 0.0) {  } 
;
\node [b] at (1.0, 0.0) {  } 
;
}}

\newcommand{\forestF}{
\tikz[planar forest ] {

\node [b] at (0.0, 0.0) {  } 
child {node [lc] at (0.0, 1.0) { 1 }  
child {node [lc] at (0.0, 1.0) { 1 }  
}
}
;
}}

\newcommand{\forestG}{
\tikz[planar forest ] {

\node [lc] at (0.0, 0.0) { 1 } 
child {node [b] at (0.0, 1.0) {  }  
child {node [lc] at (0.0, 1.0) { 1 }  
}
}
;
}}

\newcommand{\forestH}{
\tikz[planar forest ] {

\node [lc] at (0.0, 0.0) { 1 } 
child {node [lc] at (0.0, 1.0) { 1 }  
child {node [b] at (0.0, 1.0) {  }  
}
}
;
}}

\newcommand{\forestI}{
\tikz[planar forest ] {

\node [b] at (0.0, 0.0) {  } 
child {node [lc] at (-0.5, 1.0) { 1 }  
}
child {node [lc] at (0.5, 1.0) { 1 }  
}
;
}}

\newcommand{\forestJ}{
\tikz[planar forest ] {

\node [lc] at (0.0, 0.0) { 1 } 
child {node [b] at (-0.5, 1.0) {  }  
}
child {node [lc] at (0.5, 1.0) { 1 }  
}
;
}}

\newcommand{\forestK}{
\tikz[planar forest ] {

\node [b] at (0.0, 0.0) {  } 
child {node [lc] at (0.0, 1.0) { 1 }  
}
;
\node [lc] at (1.0, 0.0) { 1 } 
;
}}

\newcommand{\forestL}{
\tikz[planar forest ] {

\node [lc] at (0.0, 0.0) { 1 } 
child {node [b] at (0.0, 1.0) {  }  
}
;
\node [lc] at (1.0, 0.0) { 1 } 
;
}}

\newcommand{\forestM}{
\tikz[planar forest ] {

\node [lc] at (0.0, 0.0) { 1 } 
child {node [lc] at (0.0, 1.0) { 1 }  
}
;
\node [b] at (1.0, 0.0) {  } 
;
}}

\newcommand{\forestN}{
\tikz[planar forest ] {

\node [b] at (0.0, 0.0) {  } 
;
\node [lc] at (1.0, 0.0) { 1 } 
;
\node [lc] at (2.0, 0.0) { 1 } 
;
}}

\newcommand{\forestO}{
\tikz[planar forest ] {

\node [lc] at (0.0, 0.0) { 1 } 
child {node [lc] at (0.0, 1.0) { 1 }  
child {node [lc] at (0.0, 1.0) { 2 }  
child {node [lc] at (0.0, 1.0) { 2 }  
}
}
}
;
}}

\newcommand{\forestP}{
\tikz[planar forest ] {

\node [lc] at (0.0, 0.0) { 1 } 
child {node [lc] at (0.0, 1.0) { 2 }  
child {node [lc] at (0.0, 1.0) { 1 }  
child {node [lc] at (0.0, 1.0) { 2 }  
}
}
}
;
}}

\newcommand{\forestQ}{
\tikz[planar forest ] {

\node [lc] at (0.0, 0.0) { 1 } 
child {node [lc] at (0.0, 1.0) { 2 }  
child {node [lc] at (0.0, 1.0) { 2 }  
child {node [lc] at (0.0, 1.0) { 1 }  
}
}
}
;
}}

\newcommand{\forestR}{
\tikz[planar forest ] {

\node [lc] at (0.0, 0.0) { 1 } 
child {node [lc] at (0.0, 1.0) { 1 }  
child {node [lc] at (-0.5, 1.0) { 2 }  
}
child {node [lc] at (0.5, 1.0) { 2 }  
}
}
;
}}

\newcommand{\forestS}{
\tikz[planar forest ] {

\node [lc] at (0.0, 0.0) { 1 } 
child {node [lc] at (0.0, 1.0) { 2 }  
child {node [lc] at (-0.5, 1.0) { 1 }  
}
child {node [lc] at (0.5, 1.0) { 2 }  
}
}
;
}}

\newcommand{\forestT}{
\tikz[planar forest ] {

\node [lc] at (0.0, 0.0) { 1 } 
child {node [lc] at (-0.5, 1.0) { 1 }  
}
child {node [lc] at (0.5, 1.0) { 2 }  
child {node [lc] at (0.0, 1.0) { 2 }  
}
}
;
}}

\newcommand{\forestU}{
\tikz[planar forest ] {

\node [lc] at (0.0, 0.0) { 1 } 
child {node [lc] at (-0.5, 1.0) { 2 }  
}
child {node [lc] at (0.5, 1.0) { 1 }  
child {node [lc] at (0.0, 1.0) { 2 }  
}
}
;
}}

\newcommand{\forestV}{
\tikz[planar forest ] {

\node [lc] at (0.0, 0.0) { 1 } 
child {node [lc] at (-0.5, 1.0) { 2 }  
}
child {node [lc] at (0.5, 1.0) { 2 }  
child {node [lc] at (0.0, 1.0) { 1 }  
}
}
;
}}

\newcommand{\forestW}{
\tikz[planar forest ] {

\node [lc] at (0.0, 0.0) { 1 } 
child {node [lc] at (-1.0, 1.0) { 1 }  
}
child {node [lc] at (0.0, 1.0) { 2 }  
}
child {node [lc] at (1.0, 1.0) { 2 }  
}
;
}}

\newcommand{\forestX}{
\tikz[planar forest ] {

\node [lc] at (0.0, 0.0) { 1 } 
child {node [lc] at (0.0, 1.0) { 2 }  
child {node [lc] at (0.0, 1.0) { 2 }  
}
}
;
\node [lc] at (1.0, 0.0) { 1 } 
;
}}

\newcommand{\forestY}{
\tikz[planar forest ] {

\node [lc] at (0.0, 0.0) { 2 } 
child {node [lc] at (0.0, 1.0) { 1 }  
child {node [lc] at (0.0, 1.0) { 2 }  
}
}
;
\node [lc] at (1.0, 0.0) { 1 } 
;
}}

\newcommand{\forestAB}{
\tikz[planar forest ] {

\node [lc] at (0.0, 0.0) { 2 } 
child {node [lc] at (0.0, 1.0) { 2 }  
child {node [lc] at (0.0, 1.0) { 1 }  
}
}
;
\node [lc] at (1.0, 0.0) { 1 } 
;
}}

\newcommand{\forestBB}{
\tikz[planar forest ] {

\node [lc] at (0.0, 0.0) { 1 } 
child {node [lc] at (-0.5, 1.0) { 2 }  
}
child {node [lc] at (0.5, 1.0) { 2 }  
}
;
\node [lc] at (1.0, 0.0) { 1 } 
;
}}

\newcommand{\forestCB}{
\tikz[planar forest ] {

\node [lc] at (0.0, 0.0) { 2 } 
child {node [lc] at (-0.5, 1.0) { 1 }  
}
child {node [lc] at (0.5, 1.0) { 2 }  
}
;
\node [lc] at (1.0, 0.0) { 1 } 
;
}}

\newcommand{\forestDB}{
\tikz[planar forest ] {

\node [lc] at (0.0, 0.0) { 1 } 
child {node [lc] at (0.0, 1.0) { 1 }  
}
;
\node [lc] at (1.0, 0.0) { 2 } 
child {node [lc] at (0.0, 1.0) { 2 }  
}
;
}}

\newcommand{\forestEB}{
\tikz[planar forest ] {

\node [lc] at (0.0, 0.0) { 1 } 
child {node [lc] at (0.0, 1.0) { 2 }  
}
;
\node [lc] at (1.0, 0.0) { 1 } 
child {node [lc] at (0.0, 1.0) { 2 }  
}
;
}}

\newcommand{\forestFB}{
\tikz[planar forest ] {

\node [lc] at (0.0, 0.0) { 1 } 
child {node [lc] at (0.0, 1.0) { 2 }  
}
;
\node [lc] at (1.0, 0.0) { 2 } 
child {node [lc] at (0.0, 1.0) { 1 }  
}
;
}}

\newcommand{\forestGB}{
\tikz[planar forest ] {

\node [lc] at (0.0, 0.0) { 1 } 
child {node [lc] at (0.0, 1.0) { 1 }  
}
;
\node [lc] at (1.0, 0.0) { 2 } 
;
\node [lc] at (2.0, 0.0) { 2 } 
;
}}

\newcommand{\forestHB}{
\tikz[planar forest ] {

\node [lc] at (0.0, 0.0) { 1 } 
child {node [lc] at (0.0, 1.0) { 2 }  
}
;
\node [lc] at (1.0, 0.0) { 1 } 
;
\node [lc] at (2.0, 0.0) { 2 } 
;
}}

\newcommand{\forestIB}{
\tikz[planar forest ] {

\node [lc] at (0.0, 0.0) { 1 } 
;
\node [lc] at (1.0, 0.0) { 1 } 
;
\node [lc] at (2.0, 0.0) { 2 } 
;
\node [lc] at (3.0, 0.0) { 2 } 
;
}}

\newcommand{\forestJB}{
\tikz[planar forest ] {

\node [lc] at (0.0, 0.0) { 1 } 
child {node [lc] at (0.0, 1.0) { 1 }  
child {node [lc] at (0.0, 1.0) { 1 }  
child {node [lc] at (0.0, 1.0) { 1 }  
}
}
}
;
}}

\newcommand{\forestKB}{
\tikz[planar forest ] {

\node [lc] at (0.0, 0.0) { 1 } 
child {node [lc] at (0.0, 1.0) { 1 }  
child {node [lc] at (-0.5, 1.0) { 1 }  
}
child {node [lc] at (0.5, 1.0) { 1 }  
}
}
;
}}

\newcommand{\forestLB}{
\tikz[planar forest ] {

\node [lc] at (0.0, 0.0) { 1 } 
child {node [lc] at (-0.5, 1.0) { 1 }  
}
child {node [lc] at (0.5, 1.0) { 1 }  
child {node [lc] at (0.0, 1.0) { 1 }  
}
}
;
}}

\newcommand{\forestMB}{
\tikz[planar forest ] {

\node [lc] at (0.0, 0.0) { 1 } 
child {node [lc] at (-1.0, 1.0) { 1 }  
}
child {node [lc] at (0.0, 1.0) { 1 }  
}
child {node [lc] at (1.0, 1.0) { 1 }  
}
;
}}

\newcommand{\forestNB}{
\tikz[planar forest ] {

\node [lc] at (0.0, 0.0) { 1 } 
child {node [lc] at (0.0, 1.0) { 1 }  
child {node [lc] at (0.0, 1.0) { 1 }  
}
}
;
\node [lc] at (1.0, 0.0) { 1 } 
;
}}

\newcommand{\forestOB}{
\tikz[planar forest ] {

\node [lc] at (0.0, 0.0) { 1 } 
child {node [lc] at (-0.5, 1.0) { 1 }  
}
child {node [lc] at (0.5, 1.0) { 1 }  
}
;
\node [lc] at (1.0, 0.0) { 1 } 
;
}}

\newcommand{\forestPB}{
\tikz[planar forest ] {

\node [lc] at (0.0, 0.0) { 1 } 
child {node [lc] at (0.0, 1.0) { 1 }  
}
;
\node [lc] at (1.0, 0.0) { 1 } 
child {node [lc] at (0.0, 1.0) { 1 }  
}
;
}}

\newcommand{\forestQB}{
\tikz[planar forest ] {

\node [lc] at (0.0, 0.0) { 1 } 
child {node [lc] at (0.0, 1.0) { 1 }  
}
;
\node [lc] at (1.0, 0.0) { 1 } 
;
\node [lc] at (2.0, 0.0) { 1 } 
;
}}

\newcommand{\forestRB}{
\tikz[planar forest ] {

\node [lc] at (0.0, 0.0) { 1 } 
;
\node [lc] at (1.0, 0.0) { 1 } 
;
\node [lc] at (2.0, 0.0) { 1 } 
;
\node [lc] at (3.0, 0.0) { 1 } 
;
}}

\newcommand{\forestSB}{
\tikz[planar forest ] {

\node [w,label={ [label distance=-1mm]180:{ \scriptsize i } }] at (0.0, 0.0) {  } 
child {node [w,label={ [label distance=-1mm]180:{ \scriptsize j } }] at (0.0, 1.0) {  }  
}
;
\node [w,label={ [label distance=-1mm]0:{ \scriptsize k } }] at (1.0, 0.0) {  } 
;
\node [w,label={ [label distance=-1mm]0:{ \scriptsize l } }] at (2.0, 0.0) {  } 
;
}}

\newcommand{\forestTB}{
\tikz[planar forest ] {

\node [w,label={ [label distance=-1mm]180:{ \scriptsize i } }] at (0.0, 0.0) {  } 
child {node [w,label={ [label distance=-1mm]180:{ \scriptsize j } }] at (0.0, 1.0) {  }  
}
;
\node [w,label={ [label distance=-1mm]0:{ \scriptsize k } }] at (1.0, 0.0) {  } 
;
\node [w,label={ [label distance=-1mm]0:{ \scriptsize l } }] at (2.0, 0.0) {  } 
;
}}

\newcommand{\forestUB}{
\tikz[planar forest ] {

\node [lc] at (0.0, 0.0) { 1 } 
child {node [lc] at (0.0, 1.0) { 1 }  
}
;
\node [lc] at (1.0, 0.0) { 2 } 
;
\node [lc] at (2.0, 0.0) { 2 } 
;
}}

\newcommand{\forestVB}{
\tikz[planar forest ] {

\node [lc] at (0.0, 0.0) { 1 } 
child {node [lc] at (0.0, 1.0) { 2 }  
}
;
\node [lc] at (1.0, 0.0) { 1 } 
;
\node [lc] at (2.0, 0.0) { 2 } 
;
}}

\newcommand{\forestWB}{
\tikz[planar forest ] {

\node [lc] at (0.0, 0.0) { 1 } 
child {node [lc] at (0.0, 1.0) { 1 }  
}
;
\node [lc] at (1.0, 0.0) { 1 } 
;
\node [lc] at (2.0, 0.0) { 1 } 
;
}}

\newcommand{\forestXB}{
\tikz[planar forest ] {

\node [lc] at (0.0, 0.0) { 1 } 
;
\node [lc] at (1.0, 0.0) { 1 } 
;
}}

\newcommand{\forestYB}{
\tikz[planar forest ] {

\node [lc] at (0.0, 0.0) { 1 } 
;
\node [lc] at (1.0, 0.0) { 1 } 
child {node [lc] at (0.0, 1.0) { 1 }  
child {node [lc] at (0.0, 1.0) { 1 }  
}
}
;
}}

\newcommand{\forestAC}{
\tikz[planar forest ] {

\node [lc] at (0.0, 0.0) { 1 } 
child {node [lc] at (0.0, 1.0) { 1 }  
}
;
\node [lc] at (1.0, 0.0) { 1 } 
child {node [lc] at (0.0, 1.0) { 1 }  
}
;
}}

\newcommand{\forestBC}{
\tikz[planar forest ] {

\node [lc] at (0.0, 0.0) { 1 } 
child {node [lc] at (0.0, 1.0) { 1 }  
child {node [lc] at (0.0, 1.0) { 1 }  
}
}
;
\node [lc] at (1.0, 0.0) { 1 } 
child {node [lc] at (0.0, 1.0) { 1 }  
child {node [lc] at (0.0, 1.0) { 1 }  
}
}
;
}}

\newcommand{\forestCC}{
\tikz[planar forest ] {

\node [lc] at (0.0, 0.0) { 1 } 
child {node [lc] at (0.0, 1.0) { 1 }  
child {node [lc] at (0.0, 1.0) { 1 }  
}
}
;
\node [lc] at (1.0, 0.0) { 1 } 
;
}}

\newcommand{\forestDC}{
\tikz[planar forest ] {

\node [lc] at (0.0, 0.0) { 1 } 
child {node [lc] at (-0.5, 1.0) { 1 }  
}
child {node [lc] at (0.5, 1.0) { 1 }  
}
;
\node [lc] at (1.0, 0.0) { 1 } 
;
}}

\newcommand{\forestEC}{
\tikz[planar forest ] {

\node [lc] at (0.0, 0.0) { 1 } 
child {node [lc] at (0.0, 1.0) { 1 }  
child {node [lc] at (0.0, 1.0) { 1 }  
child {node [lc] at (0.0, 1.0) { 1 }  
}
}
}
;
}}

\newcommand{\forestFC}{
\tikz[planar forest ] {

\node [lc] at (0.0, 0.0) { 1 } 
child {node [lc] at (0.0, 1.0) { 1 }  
child {node [lc] at (0.0, 1.0) { 1 }  
}
}
;
\node [lc] at (1.0, 0.0) { 1 } 
;
}}

\newcommand{\forestGC}{
\tikz[planar forest ] {

\node [lc] at (0.0, 0.0) { 1 } 
child {node [lc] at (-0.5, 1.0) { 1 }  
}
child {node [lc] at (0.5, 1.0) { 1 }  
}
;
\node [lc] at (1.0, 0.0) { 1 } 
;
}}

\newcommand{\forestHC}{
\tikz[planar forest ] {

\node [b,label={ [label distance=-1mm]270:{ \scriptsize 1 } }] at (0.0, 0.0) {  } 
;
\node [b,label={ [label distance=-1mm]270:{ \scriptsize 2 } }] at (1.0, 0.0) {  } 
child {node [b,label={ [label distance=-1mm]90:{ \scriptsize 4 } }] at (-0.5, 1.0) {  }  
}
child {node [b,label={ [label distance=-1mm]90:{ \scriptsize 5 } }] at (0.5, 1.0) {  }  
}
;
\node [b,label={ [label distance=-1mm]270:{ \scriptsize 3 } }] at (2.5, 0.0) {  } 
child {node [b,label={ [label distance=-1mm]90:{ \scriptsize 6 } }] at (0.0, 1.0) {  }  
}
;
}}

\newcommand{\forestIC}{
\tikz[planar forest ] {

\node [b] at (0.0, 0.0) {  } 
child {node [lc] at (0.0, 1.0) { 1 }  
}
;
\node [b] at (1.0, 0.0) {  } 
;
\node [lc] at (2.0, 0.0) { 1 } 
;
}}

\newcommand{\forestJC}{
\tikz[planar forest ] {

\node [lc] at (0.0, 0.0) { 1 } 
child {node [lc] at (0.0, 1.0) { 1 }  
}
;
\node [lc] at (1.0, 0.0) { 1 } 
;
\node [lc] at (2.0, 0.0) { 1 } 
;
}}

\newcommand{\forestKC}{
\tikz[planar forest ] {

\node [lc] at (0.0, 0.0) { 1 } 
child {node [lc] at (0.0, 1.0) { 1 }  
}
;
\node [lc] at (1.0, 0.0) { 2 } 
;
\node [lc] at (2.0, 0.0) { 2 } 
;
}}

\newcommand{\forestLC}{
\tikz[planar forest ] {

\node [lc] at (0.0, 0.0) { 1 } 
child {node [lc] at (0.0, 1.0) { 2 }  
}
;
\node [lc] at (1.0, 0.0) { 1 } 
;
\node [lc] at (2.0, 0.0) { 2 } 
;
}}

\newcommand{\forestMC}{
\tikz[planar forest ] {

\node [lc] at (0.0, 0.0) { 2 } 
child {node [lc] at (0.0, 1.0) { 1 }  
}
;
\node [lc] at (1.0, 0.0) { 2 } 
;
\node [lc] at (2.0, 0.0) { 1 } 
;
}}

\newcommand{\forestNC}{
\tikz[planar forest ] {

\node [lc] at (0.0, 0.0) { 1 } 
child {node [lc] at (0.0, 1.0) { 2 }  
}
;
\node [lc] at (1.0, 0.0) { 2 } 
;
\node [lc] at (2.0, 0.0) { 1 } 
;
}}

\newcommand{\forestOC}{
\tikz[planar forest ] {

\node [b] at (0.0, 0.0) {  } 
child {node [lc] at (-0.5, 1.0) { 1 }  
}
child {node [lc] at (0.5, 1.0) { 1 }  
child {node [lc] at (-0.5, 1.0) { 2 }  
}
child {node [lc] at (0.5, 1.0) { 2 }  
}
}
;
}}

\newcommand{\forestPC}{
\tikz[planar forest ] {

\node [lc] at (0.0, 0.0) { 1 } 
child {node [lc] at (0.0, 1.0) { 2 }  
child {node [b] at (-1.0, 1.0) {  }  
}
child {node [lc] at (0.0, 1.0) { 1 }  
}
child {node [lc] at (1.0, 1.0) { 2 }  
}
}
;
}}

\newcommand{\forestQC}{
\tikz[planar forest ] {

\node [b] at (0.0, 0.0) {  } 
child {node [lc] at (-0.5, 1.0) { 3 }  
}
child {node [lc] at (0.5, 1.0) { 3 }  
child {node [lc] at (-0.5, 1.0) { 4 }  
}
child {node [lc] at (0.5, 1.0) { 4 }  
}
}
;
\node [lc] at (3.0, 0.0) { 1 } 
child {node [lc] at (0.0, 1.0) { 2 }  
child {node [b] at (-1.0, 1.0) {  }  
}
child {node [lc] at (0.0, 1.0) { 1 }  
}
child {node [lc] at (1.0, 1.0) { 2 }  
}
}
;
}}

\newcommand{\forestRC}{
\tikz[planar forest ] {

\node [b] at (0.0, 0.0) {  } 
;
}}

\newcommand{\forestSC}{
\tikz[planar forest ] {

\node [b] at (0.0, 0.0) {  } 
;
}}

\newcommand{\forestTC}{
\tikz[planar forest ] {

\node [b] at (0.0, 0.0) {  } 
;
\node [b] at (1.0, 0.0) {  } 
;
}}

\newcommand{\forestUC}{
\tikz[planar forest ] {

\node [b] at (0.0, 0.0) {  } 
child {node [b] at (0.0, 1.0) {  }  
}
;
}}

\newcommand{\forestVC}{
\tikz[planar forest ] {

\node [lc] at (0.0, 0.0) { 1 } 
;
\node [lc] at (1.0, 0.0) { 1 } 
;
}}

\newcommand{\forestWC}{
\tikz[planar forest ] {

\node [b] at (0.0, 0.0) {  } 
;
}}

\newcommand{\forestXC}{
\tikz[planar forest ] {

\node [b] at (0.0, 0.0) {  } 
child {node [lc] at (-0.5, 1.0) { 1 }  
}
child {node [lc] at (0.5, 1.0) { 1 }  
}
;
}}

\newcommand{\forestYC}{
\tikz[planar forest ] {

\node [lc] at (0.0, 0.0) { 1 } 
;
\node [b] at (1.0, 0.0) {  } 
child {node [lc] at (0.0, 1.0) { 1 }  
}
;
}}

\newcommand{\forestAD}{
\tikz[planar forest ] {

\node [lc] at (0.0, 0.0) { 1 } 
;
\node [lc] at (1.0, 0.0) { 1 } 
;
\node [b] at (2.0, 0.0) {  } 
;
}}

\newcommand{\forestBD}{
\tikz[planar forest ] {

\node [lc] at (0.0, 0.0) { 1 } 
;
\node [lc] at (1.0, 0.0) { 1 } 
;
}}

\newcommand{\forestCD}{
\tikz[planar forest ] {

\node [lc] at (0.0, 0.0) { 1 } 
;
\node [lc] at (1.0, 0.0) { 1 } 
;
}}

\newcommand{\forestDD}{
\tikz[planar forest ] {

\node [lc] at (0.0, 0.0) { 1 } 
;
\node [lc] at (1.0, 0.0) { 1 } 
child {node [lc] at (-0.5, 1.0) { 2 }  
}
child {node [lc] at (0.5, 1.0) { 2 }  
}
;
}}

\newcommand{\forestED}{
\tikz[planar forest ] {

\node [lc] at (0.0, 0.0) { 1 } 
child {node [lc] at (0.0, 1.0) { 2 }  
}
;
\node [lc] at (1.0, 0.0) { 1 } 
child {node [lc] at (0.0, 1.0) { 2 }  
}
;
}}

\newcommand{\forestFD}{
\tikz[planar forest ] {

\node [lc] at (0.0, 0.0) { 2 } 
;
\node [lc] at (1.0, 0.0) { 1 } 
;
\node [lc] at (2.0, 0.0) { 1 } 
child {node [lc] at (0.0, 1.0) { 2 }  
}
;
}}

\newcommand{\forestGD}{
\tikz[planar forest ] {

\node [lc] at (0.0, 0.0) { 2 } 
;
\node [lc] at (1.0, 0.0) { 2 } 
;
\node [lc] at (2.0, 0.0) { 1 } 
;
\node [lc] at (3.0, 0.0) { 1 } 
;
}}

\newcommand{\forestHD}{
\tikz[planar forest ] {

\node [b] at (0.0, 0.0) {  } 
child {node [lc] at (0.0, 1.0) { 1 }  
}
;
\node [lc] at (1.0, 0.0) { 1 } 
;
\node [lc] at (2.0, 0.0) { 2 } 
child {node [b] at (-0.5, 1.0) {  }  
}
child {node [lc] at (0.5, 1.0) { 2 }  
}
;
}}

\newcommand{\forestID}{
\tikz[planar forest ] {

\node [b] at (0.0, 0.0) {  } 
child {node [lc] at (0.0, 1.0) { 1 }  
}
;
\node [lc] at (1.0, 0.0) { 1 } 
;
\node [lc] at (2.0, 0.0) { 2 } 
child {node [b] at (-0.5, 1.0) {  }  
}
child {node [lc] at (0.5, 1.0) { 2 }  
}
;
}}

\newcommand{\forestJD}{
\tikz[planar forest ] {

\node [b] at (0.0, 0.0) {  } 
child {node [lc] at (0.0, 1.0) { 1 }  
}
;
\node [lc] at (1.0, 0.0) { 1 } 
;
}}

\newcommand{\forestKD}{
\tikz[planar forest ] {

\node [lc] at (0.0, 0.0) { 2 } 
child {node [b] at (-0.5, 1.0) {  }  
}
child {node [lc] at (0.5, 1.0) { 2 }  
}
;
}}

\newcommand{\forestLD}{
\tikz[planar forest ] {

\node [lc] at (0.0, 0.0) { 2 } 
child {node [b] at (-0.5, 1.0) {  }  
}
child {node [lc] at (0.5, 1.0) { 2 }  
}
;
}}

\newcommand{\forestMD}{
\tikz[planar forest ] {

\node [b] at (0.0, 0.0) {  } 
child {node [lc] at (0.0, 1.0) { 1 }  
}
;
\node [lc] at (1.0, 0.0) { 1 } 
;
}}

\newcommand{\forestND}{
\tikz[planar forest ] {

\node [b] at (0.0, 0.0) {  } 
child {node [lc] at (0.0, 1.0) { 1 }  
}
;
\node [lc] at (1.0, 0.0) { 1 } 
;
\node [lc] at (2.0, 0.0) { 2 } 
child {node [b] at (-0.5, 1.0) {  }  
}
child {node [lc] at (0.5, 1.0) { 2 }  
}
;
}}

\newcommand{\forestOD}{
\tikz[planar forest ] {

\node [lc] at (0.0, 0.0) { 1 } 
;
\node [lc] at (1.0, 0.0) { 1 } 
;
}}

\newcommand{\forestPD}{
\tikz[planar forest ] {

\node [lc] at (0.0, 0.0) { 1 } 
child {node [lc] at (0.0, 1.0) { 2 }  
}
;
\node [lc] at (1.0, 0.0) { 1 } 
child {node [lc] at (0.0, 1.0) { 2 }  
}
;
}}

\newcommand{\forestQD}{
\tikz[planar forest ] {

\node [b] at (0.0, 0.0) {  } 
;
}}

\newcommand{\forestRD}{
\tikz[planar forest ] {

\node [lc] at (0.0, 0.0) { 1 } 
;
\node [lc] at (1.0, 0.0) { 1 } 
;
}}

\newcommand{\forestSD}{
\tikz[planar forest ] {

\node [b] at (0.0, 0.0) {  } 
;
}}

\newcommand{\forestTD}{
\tikz[planar forest ] {

\node [lc] at (0.0, 0.0) { 1 } 
;
\node [lc] at (1.0, 0.0) { 1 } 
;
}}

\newcommand{\forestUD}{
\tikz[planar forest ] {

\node [lc] at (0.0, 0.0) { 1 } 
child {node [lc] at (0.0, 1.0) { 1 }  
}
;
}}

\newcommand{\forestVD}{
\tikz[planar forest ] {

\node [b] at (0.0, 0.0) {  } 
;
}}

\newcommand{\forestWD}{
\tikz[planar forest ] {

\node [lc] at (0.0, 0.0) { 1 } 
;
\node [lc] at (1.0, 0.0) { 1 } 
;
}}

\newcommand{\forestXD}{
\tikz[planar forest ] {

\node [b] at (0.0, 0.0) {  } 
;
}}

\newcommand{\forestYD}{
\tikz[planar forest ] {

\node [lc] at (0.0, 0.0) { 1 } 
;
\node [lc] at (1.0, 0.0) { 1 } 
;
}}

\newcommand{\forestAE}{
\tikz[planar forest ] {

\node [lc] at (0.0, 0.0) { 1 } 
child {node [lc] at (0.0, 1.0) { 1 }  
}
;
}}

\newcommand{\forestBE}{
\tikz[planar forest ] {

\node [b] at (0.0, 0.0) {  } 
;
}}

\newcommand{\forestCE}{
\tikz[planar forest ] {

\node [lc] at (0.0, 0.0) { 1 } 
;
\node [lc] at (1.0, 0.0) { 1 } 
;
}}

\newcommand{\forestDE}{
\tikz[planar forest ] {

\node [b] at (0.0, 0.0) {  } 
;
}}

\newcommand{\forestEE}{
\tikz[planar forest ] {

\node [lc] at (0.0, 0.0) { 1 } 
;
\node [lc] at (1.0, 0.0) { 1 } 
;
}}

\newcommand{\forestFE}{
\tikz[planar forest ] {

\node [b] at (0.0, 0.0) {  } 
child {node [b] at (0.0, 1.0) {  }  
}
;
}}

\newcommand{\forestGE}{
\tikz[planar forest ] {

\node [b] at (0.0, 0.0) {  } 
;
\node [b] at (1.0, 0.0) {  } 
;
}}

\newcommand{\forestHE}{
\tikz[planar forest ] {

\node [b] at (0.0, 0.0) {  } 
;
}}

\newcommand{\forestIE}{
\tikz[planar forest ] {

\node [lc] at (0.0, 0.0) { 1 } 
;
\node [lc] at (1.0, 0.0) { 1 } 
;
}}

\newcommand{\forestJE}{
\tikz[planar forest ] {

\node [lc] at (0.0, 0.0) { 1 } 
child {node [lc] at (0.0, 1.0) { 1 }  
}
;
}}

\newcommand{\forestKE}{
\tikz[planar forest ] {

\node [b] at (0.0, 0.0) {  } 
child {node [b] at (0.0, 1.0) {  }  
}
;
}}

\newcommand{\forestLE}{
\tikz[planar forest ] {

\node [b] at (0.0, 0.0) {  } 
;
\node [b] at (1.0, 0.0) {  } 
;
}}

\newcommand{\forestME}{
\tikz[planar forest ] {

\node [b,label={ [label distance=-1mm]270:{ \scriptsize a } }] at (0.0, 0.0) {  } 
child {node [lc,label={ [label distance=-1mm]180:{ \scriptsize {b} } }] at (-1.0, 1.0) { 1 }  
child {node [lc,label={ [label distance=-1mm]180:{ \scriptsize {c} } }] at (0.0, 1.0) { 2 }  
}
}
child {node [lc,label={ [label distance=-1mm]90:{ \scriptsize {d} } }] at (0.0, 1.0) { 3 }  
}
child {node [b,label={ [label distance=-1mm]90:{ \scriptsize {e} } }] at (1.0, 1.0) {  }  
}
;
\node [lc,label={ [label distance=-1mm]270:{ \scriptsize f } }] at (2.0, 0.0) { 2 } 
child {node [b,label={ [label distance=-1mm]90:{ \scriptsize {g} } }] at (0.0, 1.0) {  }  
}
;
\node [lc,label={ [label distance=-1mm]270:{ \scriptsize h } }] at (3.0, 0.0) { 3 } 
child {node [lc,label={ [label distance=-1mm]90:{ \scriptsize {i} } }] at (0.0, 1.0) { 1 }  
}
;
}}

\newcommand{\forestNE}{
\tikz[planar forest ] {

\node [lc] at (0.0, 0.0) { 1 } 
child {node [lc] at (0.0, 1.0) { 1 }  
}
;
\node [lc] at (1.0, 0.0) { 1 } 
;
\node [lc] at (2.0, 0.0) { 1 } 
;
}}

\newcommand{\forestOE}{
\tikz[planar forest ] {

\node [lc] at (0.0, 0.0) { 1 } 
child {node [lc] at (0.0, 1.0) { 2 }  
}
;
\node [lc] at (1.0, 0.0) { 1 } 
;
\node [lc] at (2.0, 0.0) { 2 } 
;
}}

\newcommand{\forestPE}{
\tikz[planar forest ] {

\node [lc] at (0.0, 0.0) { 1 } 
child {node [lc] at (0.0, 1.0) { 2 }  
}
;
\node [lc] at (1.0, 0.0) { 2 } 
;
\node [lc] at (2.0, 0.0) { 1 } 
;
}}

\newcommand{\forestQE}{
\tikz[planar forest ] {

\node [lc] at (0.0, 0.0) { 1 } 
child {node [lc] at (0.0, 1.0) { 1 }  
}
;
\node [lc] at (1.0, 0.0) { 2 } 
;
\node [lc] at (2.0, 0.0) { 2 } 
;
}}

\newcommand{\forestRE}{
\tikz[planar forest ] {

\node [lc] at (0.0, 0.0) { 1 } 
child {node [lc] at (0.0, 1.0) { 2 }  
}
;
\node [lc] at (1.0, 0.0) { 1 } 
;
\node [lc] at (2.0, 0.0) { 2 } 
;
}}

\newcommand{\forestSE}{
\tikz[planar forest ] {

\node [lc] at (0.0, 0.0) { 1 } 
child {node [lc] at (0.0, 1.0) { 1 }  
}
;
\node [lc] at (1.0, 0.0) { 2 } 
;
\node [lc] at (2.0, 0.0) { 2 } 
;
}}

\newcommand{\forestTE}{
\tikz[planar forest ] {

\node [b] at (0.0, 0.0) {  } 
child {node [lc] at (-1.0, 1.0) { 1 }  
}
child {node [lc] at (0.0, 1.0) { 2 }  
child {node [lc] at (0.0, 1.0) { 1 }  
}
}
child {node [b] at (1.0, 1.0) {  }  
child {node [lc] at (0.0, 1.0) { 2 }  
}
}
;
}}

\newcommand{\forestUE}{
\tikz[planar forest ] {

\node [b] at (0.0, 0.0) {  } 
child {node [lc] at (-1.0, 1.0) { 1 }  
}
child {node [lc] at (0.0, 1.0) { 2 }  
child {node [lc] at (0.0, 1.0) { 1 }  
}
}
child {node [b] at (1.0, 1.0) {  }  
child {node [lc] at (0.0, 1.0) { 2 }  
}
}
;
}}

\newcommand{\forestVE}{
\tikz[planar forest ] {

\node [lc] at (0.0, 0.0) { 1 } 
;
\node [lc] at (1.0, 0.0) { 1 } 
;
}}

\newcommand{\forestWE}{
\tikz[planar forest ] {

\node [b] at (0.0, 0.0) {  } 
child {node [lc] at (-0.5, 1.0) { 2 }  
}
child {node [b] at (0.5, 1.0) {  }  
child {node [lc] at (0.0, 1.0) { 2 }  
}
}
;
}}

\newcommand{\forestXE}{
\tikz[planar forest ] {

\node [lc] at (0.0, 0.0) { 1 } 
;
\node [lc] at (1.0, 0.0) { 2 } 
child {node [lc] at (0.0, 1.0) { 1 }  
}
;
\node [lc] at (2.0, 0.0) { 2 } 
;
}}

\newcommand{\forestYE}{
\tikz[planar forest ] {

\node [b] at (0.0, 0.0) {  } 
child {node [b] at (0.0, 1.0) {  }  
}
;
}}

\newcommand{\forestAF}{
\tikz[planar forest ] {

\node [lc] at (0.0, 0.0) { 1 } 
;
\node [lc] at (1.0, 0.0) { 2 } 
child {node [lc] at (0.0, 1.0) { 1 }  
}
;
\node [b] at (2.0, 0.0) {  } 
child {node [lc] at (0.0, 1.0) { 2 }  
}
;
}}

\newcommand{\forestBF}{
\tikz[planar forest ] {

\node [b] at (0.0, 0.0) {  } 
;
}}

\newcommand{\forestCF}{
\tikz[planar forest ] {

\node [b] at (0.0, 0.0) {  } 
child {node [lc] at (-1.0, 1.0) { 1 }  
}
child {node [lc] at (0.0, 1.0) { 2 }  
child {node [lc] at (0.0, 1.0) { 1 }  
}
}
child {node [b] at (1.0, 1.0) {  }  
child {node [lc] at (0.0, 1.0) { 2 }  
}
}
;
}}

\newcommand{\forestDF}{
\tikz[planar forest ] {

\node [b] at (0.0, 0.0) {  } 
;
}}

\newcommand{\forestEF}{
\tikz[planar forest ] {

\node [lc] at (0.0, 0.0) { 1 } 
;
\node [lc] at (1.0, 0.0) { 1 } 
;
}}

\newcommand{\forestFF}{
\tikz[planar forest ] {

\node [b] at (0.0, 0.0) {  } 
;
}}

\newcommand{\forestGF}{
\tikz[planar forest ] {

\node [lc] at (0.0, 0.0) { 1 } 
;
\node [lc] at (1.0, 0.0) { 1 } 
;
}}

\newcommand{\forestHF}{
\tikz[planar forest ] {

\node [lc] at (0.0, 0.0) { 1 } 
child {node [lc] at (0.0, 1.0) { 1 }  
}
;
}}

\usepackage{csquotes}
\usepackage{aliascnt}
\usepackage[capitalize]{cleveref}
\crefname{cond}{condition}{conditions}
\crefname{figure}{Figure}{Figures}
\DeclareMathOperator{\Adm}{Adm}

\newcommand{\arsinh}{\operatorname{arsinh}}

\allowdisplaybreaks

\newtheorem{theorem}{Theorem}[section]
\newaliascnt{defcnt}{theorem}
\crefname{defcnt}{Definition}{Definitions}
\newtheorem{definition}[theorem]{Definition}
\newtheorem*{definition*}{Definition}
\newaliascnt{propcnt}{theorem}
\crefname{propcnt}{Proposition}{Propositions}
\newtheorem{proposition}[propcnt]{Proposition}
\newaliascnt{lemcnt}{theorem}
\crefname{lemcnt}{Lemma}{Lemmas}
\newtheorem{lemma}[lemcnt]{Lemma}
\newaliascnt{remcnt}{theorem}
\crefname{remcnt}{Remark}{Remarks}
\newtheorem{remark}[remcnt]{Remark}
\newtheorem*{remark*}{Remark}
\newaliascnt{colcnt}{theorem}
\crefname{colcnt}{Corollary}{Corollaries}

\newaliascnt{asscnt}{theorem}
\crefname{asscnt}{Assumption}{Assumptions}
\newtheorem{ass}[asscnt]{Assumption}
\newtheorem*{notation*}{Notation}
\newaliascnt{excnt}{theorem}
\crefname{excnt}{Example}{Examples}
\newtheorem{ex}[excnt]{Example}
\title{
Derivation of optimal stochastic Runge-Kutta methods with exotic and decorated Butcher series for the weak integration of stochastic dynamics
}

\author{
Adrien Busnot Laurent\textsuperscript{1}, Kristian Debrabant\textsuperscript{2} and Anne Kværnø\textsuperscript{3}
}
\newcommand{\arxivpaper}[2]{#1}
\usepackage{enumitem}
\begin{document}
\footnotetext[1]{
Univ Rennes, INRIA (Research team MINGuS), IRMAR (CNRS UMR 6625) and ENS Rennes, France.
Adrien.Busnot-Laurent@inria.fr.}
\footnotetext[2]{
Department of Mathematics and Computer Science, University of Southern Denmark, Odense, Denmark.
Debrabant@imada.sdu.dk.}
\footnotetext[3]{
Department of Mathematical Sciences, Norwegian University of Science and Technology, Trondheim, Norway.
Anne.Kvarno@ntnu.no.}

\maketitle

\begin{abstract}
The design of numerical integrators for solving stochastic dynamics with high weak order relies on tedious calculations and is subject to a high number of order conditions.
The original approaches from the literature consider strong approximations and adapt them for the weak approximation by replacing the iterated stochastic integrals by appropriate random variables. The methods obtained this way are sub-optimal in their number of function evaluations and the analysis of order conditions is unnecessarily complicated.
We provide in this paper a novel approach, relying on well-chosen sets of random Runge-Kutta coefficients, that greatly reduce the number of order conditions. The approach is successfully applied to the creation of a collection of new stochastic Runge-Kutta methods of second weak order with an optimal number of function evaluations and a smaller number of random variables.
The efficiency of the new methods is confirmed with numerical experiments and a modern algebraic approach using Hopf algebras is provided for the derivation and the study of the order conditions.

\smallskip

\smallskip

\end{abstract}

\tikzexternaldisable

\def\ts{\thinspace}

\xdef\basis{2}

\newcommand{\DefineMethodName}[2]{%
  \expandafter\newcommand\csname rmk-#1\endcsname{#2}%
}
\newcommand{\Method}[1]{\csname rmk-#1\endcsname}

\DefineMethodName{AAK1}{BDK1OLD}
\DefineMethodName{AAK2}{BDK2OLD}
\DefineMethodName{AAK3}{BDK3}
\DefineMethodName{AAK4}{BDK1}
\DefineMethodName{AAK5}{BDK2}
\DefineMethodName{DRI1}{DRI1}
\DefineMethodName{RI5}{RI5}
\DefineMethodName{RI6}{RI6}
\DefineMethodName{W2Ito1}{W2Ito1}
\DefineMethodName{Platen}{Platen/RI6}

\newcommand{\initcompeffort}[1]{
\renewcommand{\do}[1]{
\pgfplotstableset{
    create on use/compeffort-##1/.style={
        create col/expr={\thisrow{g0-##1} + #1*\thisrow{g-##1} + \thisrow{RV-##1}}
    }}}
\docsvlist{AAK1,AAK2,AAK3,AAK4,AAK5,Platen,W2Ito1,RI5,RI6,DRI1}
}

\pgfplotsset{AAK1style/.style={mark=+,cyan,thick}}
\pgfplotsset{AAK2style/.style={mark=-,purple,thick}}
\pgfplotsset{AAK3style/.style={mark=star,magenta,thick}}
\pgfplotsset{AAK4style/.style={mark=o,red,thick}}
\pgfplotsset{AAK5style/.style={mark=diamond,blue,thick}}
\pgfplotsset{Platenstyle/.style={mark=square,green,thick}}
\pgfplotsset{RI6style/.style={mark=square,green,thick}}
\pgfplotsset{DRI1style/.style={mark=x,brown,thick}}
\pgfplotsset{W2Ito1style/.style={mark=*,yellow,thick}}
\pgfplotsset{RI5style/.style={mark=|,gray,thick}}

\newcommand{\adderrorvshplot}[3]{
        \addplot[only marks,#3style] table[x ={h}, y={#2-#3}]{#1};
        \addplot[#3style,mark=none,forget plot] table[x={h}, y={create col/linear regression={y={#2-#3}}}]{#1};
        \expandafter\xdef\csname slope#3\endcsname{\pgfplotstableregressiona}
        \addlegendentry{\text{\Method{#3}}, $\hat{p}=\pgfmathprintnumber{\expandafter\csname slope#3\endcsname}$}
}

\newcommand{\adderrororderplots}[3]{
\renewcommand{\do}[1] {\adderrorvshplot{#1}{#2}{##1}}
\docsvlist{#3}
}

\newcommand{\adderrorvsplots}[4]{
\renewcommand{\do}[1]{\addplot[##1style] table[x ={#2-##1}, y={#3-##1}]{#1};}
\docsvlist{#4}
}

\newcounter{MyPlotLegend}

\newcommand{\numfigure}[4]{
\begin{figure}[ht]
\stepcounter{MyPlotLegend}%
\edef\thislegendname{gemeinsameLegende-\theMyPlotLegend}%
\begin{tikzpicture}
\begin{groupplot}[group style={
        group size=2 by 1,
        horizontal sep=3.0cm
    }]
    \nextgroupplot[
        width=0.4\textwidth,
        height=0.2\textheight,
        axis x line=bottom,
        xmode = log,
        log basis x=2,
        axis y line=left,
        ymode = log,
        log basis y=2,
        xlabel=$h$,
        ylabel={${|\hat{\E}}[\phi(X_N) - \phi(X(T_N))]|$},
        legend to name=\thislegendname,
        legend style={at={(0.5,1.5)},anchor=north,/tikz/every even column/.append style={column sep=2mm}},
        legend columns=\arxivpaper{5}{3}]
        \adderrororderplots{#1}{#2}{#3}
\nextgroupplot[
        width=0.4\textwidth,
        height=0.2\textheight,
        axis x line=bottom,
        xmode = log,
        log basis x=\basis,
        axis y line=left,
        ymode = log,
        log basis y=\basis,
        title={},
        ylabel={Comp.\ effort},
        xlabel={${|\hat{\E}}[\phi(X_N) - \phi(X(T_N))]|$},
        legend pos=north east]
        \adderrorvsplots{#1}{#2}{compeffort}{#3}
\end{groupplot}

\path (group c1r1.north west) -- (group c2r1.north east)
      node[midway, above=1cm] {\pgfplotslegendfromname{\thislegendname}};

\end{tikzpicture}
\caption{#4}
\end{figure}
}

\section{Introduction}
\label{section:Introduction}

We present new approaches for the high order numerical integration in the weak sense of general Itô stochastic differential equations of the form
\begin{equation}
\label{equation:SDE_Ito}
dX(t)=f^0(X(t))dt+\sum_{p=1}^m f^p(X(t)) dW^p(t),\quad X(0)=X_0,
\end{equation}
and SDEs with Stratonovich noise
\begin{equation}
\label{equation:SDE_Strato}
dX(t)=f^0(X(t))dt+\sum_{p=1}^m f^p(X(t)) \circ dW^p(t),\quad X(0)=X_0,
\end{equation}
where the $f^p$ are smooth Lipschitz vector fields and the $W^p$ are standard independent Brownian motions defined on a probability space fulfilling the standard assumptions.

The first approaches for the numerical approximation with high weak order of \eqref{equation:SDE_Ito} and \eqref{equation:SDE_Strato} were derived from strong methods, that are, pathwise approximations. This resulted in unnecessarily complicated methods and tedious order conditions and analysis. The first stochastic Runge-Kutta methods of second weak order had a number of function evaluations in the order of $m^2$ (see \cite{Platen84zza,Kloeden92nso,tocino02wso,komori07wso,debrabant08cos}), where one could theoretically expect the optimal number of $2m+2$ function evaluations (at least in the Itô case).
The methods were simplified using a B-series approach in \cite{Rossler07sor,Rossler09sor} to obtain a number of function evaluations that is affine in $m$. The methods were further simplified in \cite{Tang17ews} in the Itô case, without reaching the optimal number of function evaluations.
In the present paper, we focus on deriving weak approximations without relying on any intermediate strong approximation, and we provide new methods with the optimal number of function evaluations for general Itô and Stratonovich SDEs \eqref{equation:SDE_Ito}-\eqref{equation:SDE_Strato}.
We consider the following general class of stochastic Runge-Kutta methods, in the spirit of the methods introduced in \cite{Rossler06rkm}:
\begin{align}
\label{equation:def_RKsto}
H_i^p&=X_n+h\sum_{j=1}^s Z^{p,0}_{i,j} f^0(H_j^0)+\sqrt{h}\sum_{j=1}^s \sum_{q=1}^m Z^{p,q}_{i,j}f^q(H_j^q),\quad \nonumber i=1,\dots,s,\\
X_{n+1}&=X_n+h\sum_{i=1}^s z^{0}_{i} f^0(H_i^0)+\sqrt{h}\sum_{i=1}^s \sum_{p=1}^m z^{p}_{i}f^p(H_i^p),
\end{align}
with the Runge-Kutta coefficients $z^{p}\in \R^s$ and $Z^{p,q}\in \R^{s,s}$ being random variables and the dependency in $n$ of the $H_i^p$ being omitted for simplicity.
We follow the standard approach for the design of high weak order methods with the Milstein methodology \cite{Milstein85wao}, which uses local estimates. Under standard bounded moments properties, one typically obtains global weak order (see, for instance, the textbooks \cite{Kloeden92nso,milstein95nio,Milstein04snf} and references therein).
The focus in this paper is on the design of new high-order methods for the weak error, with the minimal number of function evaluations, and on the development of the algebraic foundations underlying stochastic numerics of general SDEs.

The calculation of order conditions is tedious and requires the use of Butcher series techniques.
In the last decades, several works extended the standard Butcher-series \cite{Butcher72aat,Hairer74otb} (see also the textbooks~\cite{Hairer06gni,Butcher16nmf,Butcher21bsa} and the review~\cite{McLachlan17bsa}) to the stochastic context.
Burrage and Burrage~\cite{Burrage96hso,Burrage00oco} and Komori, Mitsui and Sugiura~\cite{Komori97rta} introduced stochastic B-series in the context of strong convergence.
The analysis was later extended to stochastic Runge-Kutta methods by Rö{\ss}ler~\cite{Rossler04ste,Rossler06rta,Rossler06rkm,Rossler10ste,Rossler10saw}, Komori~\cite{komori07mcr,komori07wso,komori17wso} and Debrabant and Kv{\ae}rn{\o}~\cite{Debrabant08bsa,Debrabant10rkm,Debrabant11cos,debrabant17cah,anmarkrud18goc} for the creation of high order weak and strong integrators.
We take inspiration from the previous tree formalisms and use decorated forests, where we enforce the forests to be independent of the dimension of the problem and the number of noises, so that one order condition corresponds to one forest only. The decorated forests formalism allows to represent the Taylor expansion of the numerical methods and we shall see that a smaller formalism of exotic forests is sufficient to write the Taylor expansion of the exact flow.
In the context of additive noise, the formalism of exotic series was introduced by Bronasco, Laurent, and Vilmart in \cite{Laurent20eab, Laurent21ata, Bronasco22ebs} for the creation of integrators for solving stochastic dynamics with high order in the weak sense and for the invariant measure (see also, for instance, \cite{BouRabee10lra, Leimkuhler13rco, Abdulle14hon, Abdulle15lta, Leimkuhler16tco}).
This formalism, combined with the aromatic B-series \cite{Chartier07pfi, Iserles07bsm, Bogfjellmo22uat, Laurent23tab, Laurent23tld} and the LB-series \cite{Iserles00lgm}, was extended in \cite{Laurent21ocf, Bronasco25hoi} for the extrinsic and intrinsic numerical integration of SDEs on manifolds, and analogous algebraic formalisms are now used in a variety of different fields \cite{Bronsard22aod, Deng23fdo, Bonicelli25ebs}.
The fundamental geometric and algebraic properties of the exotic formalism of trees were later studied in \cite{Bronasco22cef, Laurent23tue} (see also \cite{Chartier10aso, McLachlan16bsm, MuntheKaas16abs, Bogfjellmo19aso} in the deterministic setting).
The exotic and decorated series formalism conveniently unifies the previous forests formalisms for the study of Euclidean SDEs with multiplicative noise and proves to be a crucial tool for simplifying the tedious calculation.
We adapt the expansions of the exact and numerical flows in exotic and decorated series, uncover the combinatorial links between the two types of expansions, and use a modern combinatorial approach to study the order conditions and the algebraic properties of the Taylor expansions of flows.

The article is organised as follows.
We derive in Section \ref{section:reduced_order_conditions} the general order conditions of stochastic Runge-Kutta methods and present a new set of random Runge-Kutta coefficients that reduces considerably the number of order conditions.
A collection of new methods with optimal number of stages and function evaluations is presented in Section \ref{section:new_methods}.
We propose numerical illustrations to confirm our theoretical findings in Section \ref{section:numerical_experiments}.
The calculation of order conditions and their algebraic study with exotic and decorated forests is then detailed in Section \ref{section:general_conditions}.
We discuss outlooks and future works in Section \ref{section:conclusion}.

\section{Reduced order conditions for stochastic Runge-Kutta methods}
\label{section:reduced_order_conditions}

In order to construct methods whose coefficients have to fulfill comparatively few order conditions for weak order two, we will first derive general order conditions and then consider a well chosen ansatz for the coefficients that splits the role of the Runge-Kutta coefficients and the random variables. This will then allow us to obtain a collection of new stochastic Runge-Kutta integrators with minimal number of function evaluations.

\subsection{General order conditions for weak order two}

Let $\CC^\infty_P(\R^d)$ be the space of smooth functions $\phi$ that satisfy polynomial estimates of the form
\[
|\phi^{(k)}(x)|\leq C(1+|x|^K),\quad k=0,1,\dots.
\]
Similarly, let $\mathfrak{X}_P(\R^d)$ be the space of globally Lipschitz vector fields whose components lie in $\CC^\infty_P(\R^d)$.
In all that follows, we assume for simplicity that $f^p\in \mathfrak{X}_P(\R^d)$ and we use the following notation for the $i$-th component of the partial derivatives of $f^p$:
\[f^{p,i}_{j_1\dots j_k}(x):=\frac{\partial f^{p,i}}{\partial x_{j_1} \dots \partial x_{j_k}}(x).\]
As our approach focuses on the high order analysis of numerical integrators, we use smooth maps for simplicity and refer to \cite{Milstein04snf} for weaker assumptions.

A one-step numerical integrator $X_{n+1}=\psi_h(X_n)$ for solving \eqref{equation:SDE_Ito}/\eqref{equation:SDE_Strato} is of local weak order $p$ if for $h\leq h_0$ small enough and all $\phi \in \CC^\infty_P(\R^d)$, the following estimate is satisfied
\[
\abs{\E[\phi(X(h))|X(0)=x]-\E[\phi(X_1)|X_0=x]}\leq C(1+|x|^K) h^{p+1}.
\]
For a test function $\phi \in \CC^\infty_P(\R^d)$, a weak approximation focuses on approaching $u(x,t)=\E[\phi(X(t))|X(0)=x]$.
A classical tool for the study of equations \eqref{equation:SDE_Ito}/\eqref{equation:SDE_Strato} is the backward Kolmogorov equation (see, for instance, \cite{Hasminskii80sso, Milstein04snf, Faou09csd, Abdulle14hon, Kopec15wbea, Kopec15wbeb, Laurent21ata}).
It states that $u(x,t)$ solves the following deterministic parabolic PDE in $\R^d$,
\begin{equation}
\label{equation:Kolmogorov}
\frac{\partial u}{\partial t} = \LL u, \quad u(x,0)=\phi(x),
\quad x\in \R^d,\quad t>0,
\end{equation}
where the generator $\LL$ is given for the Itô SDE \eqref{equation:SDE_Ito} by
\[
\LL_{\text{Itô}}\phi=\sum_{i=1}^d \phi_i f^{0,i} + \frac{1}{2} \sum_{i,j=1}^d \sum_{p=1}^m \phi_{ij}(f^{p,i}, f^{p,j}),
\]
and for the Stratonovich SDE \eqref{equation:SDE_Strato} by
\[
\LL_{\text{Strato}}\phi=\sum_{i=1}^d \phi_i f^{0,i} + \frac{1}{2} \sum_{i,j=1}^d \sum_{p=1}^m \phi_{ij}(f^{p,i}, f^{p,j}) + \frac{1}{2} \sum_{i,j=1}^d \sum_{p=1}^m \phi_i f^{p,i}_j f^{p,j}.
\]
Using \eqref{equation:Kolmogorov}, the weak quantity $u(x,h)$ thus has the following expansion for all $p$:
\begin{equation}
\label{equation:exact_expansion}
u(x,h)=\phi(x)+h\LL\phi(x)+\dots+\frac{h^p}{p!}\LL^p\phi(x)+ h^{p+1} R_p\phi(x),\quad \abs{R_p\phi(x)}\leq C(1+\abs{x}^K).
\end{equation}

Let us now consider the Taylor expansion of $\E[\phi(X_1)|X_0=x]$. We assume the following natural assumption on the Runge-Kutta coefficients, which ensures that the method \eqref{equation:def_RKsto} has bounded moments of all order and that non-integer powers in their Taylor expansion vanish.
\begin{ass}
\label{ass:RV_conditions}
For any bijection $\sigma\colon \{0,\dots,m\}\rightarrow\{0,\dots,m\}$ with $\sigma(0)=0$, the coefficients of the numerical method \eqref{equation:def_RKsto} satisfy for $\alpha,\beta\in\N$, $p_{l_1},q_{l_2},r_{l_2}\in\{0,1,\dots,m\}$ and $i_{l_1},j_{l_2},k_{l_2}\in\{1,\dots,s\}$ for $l_1\in\{1,\dots,\alpha\}$, $l_2\in\{1,\dots,\beta\}$
\[
\E[z^{\sigma(p_1)}_{i_1}\dots z^{\sigma(p_\alpha)}_{i_\alpha}Z^{\sigma(q_1),\sigma(r_1)}_{j_1,k_1}\dots Z^{\sigma(q_\beta),\sigma(r_\beta)}_{j_\beta,k_\beta}]=\E[z^{p_1}_{i_1}\dots z^{p_\alpha}_{i_\alpha}Z^{q_1,r_1}_{j_1,k_1}\dots Z^{q_\beta,r_\beta}_{j_\beta,k_\beta}]<\infty,
\]
and the following moments vanish:
\[
\E[z^{p_1}_{i_1}\dots z^{p_\alpha}_{i_\alpha}Z^{q_1,r_1}_{j_1,k_1}\dots Z^{q_\beta,r_\beta}_{j_\beta,k_\beta}] = 0 \text{ if }\abs{\{p_i\neq 0\}}+\abs{\{r_i\neq 0\}} \text{ is odd.}
\]
\end{ass}
Under \cref{ass:RV_conditions}, the stochastic Runge-Kutta methods satisfy a similar expansion to \eqref{equation:exact_expansion} \cite{Milstein85wao}:
\begin{equation}
\label{equation:num_expansion}
\E[\phi(X_1)|X_0=x] = \phi(x) + h\AA_1 \phi(x)+ \dots+ h^p \AA_p\phi(x)+ h^{p+1} R_p\phi(x),
\end{equation}
where the remainder satisfies $\abs{R_p\phi(x)}\leq C(1+\abs{x}^K)$ and the $\AA_i$ are linear differential operators with coefficients depending smoothly on the $f^p$ and their partial derivatives (that is, the coefficients are polynomials in the coordinates of the infinite jet space over the $f^p$ \cite{Kolar93noi, MuntheKaas16abs}).
If the expansions \eqref{equation:exact_expansion} and \eqref{equation:num_expansion} match up to order $p$, we obtain a method of (at least) local weak order $p$.
Sufficient conditions for global weak order $p$ are given by the following result.
\begin{proposition}[\cite{Milstein85wao, Milstein04snf}]\label{prop:Convergence}
Consider a one-step integrator that has a Taylor expansion \eqref{equation:num_expansion} that satisfies
\begin{equation}
\label{equation:TT_condition}
\AA_j=\frac{\LL^j}{j!}, \quad j=1,\dots, p.
\end{equation}
Assume further that the integrator has bounded moments of any order,
\[
\sup_{n\geq 0} \E[\abs{X_n}^{2k}]<\infty.\]
Then, the method has global weak order $p$, that is, for $T>0$, for all $h\leq h_0$ small enough with $Nh=T$, for all test functions $\phi\in\CC^\infty_P(\R^d)$ and initial conditions $X_0$, there exists $C>0$ such that
\[\abs{\E[\phi(X_N)]-\E[\phi(X(T))]}\leq C h^p.\]
\end{proposition}

The derivation of order conditions is straightforward thanks to Proposition \ref{prop:Convergence}. However, it relies on the explicit knowledge of the operators $\AA_j$ and $\LL^j$. As the calculations get tedious, we use algebraic objects to derive the explicit combinatorics.
\begin{theorem}
\label{theorem:RK_conditions_general}
Consider a stochastic Runge-Kutta method of the form \eqref{equation:def_RKsto} with coefficients that satisfy \cref{ass:RV_conditions}.
Assume further that the exotic and Isserlis order conditions given in \cref{table:exotic_order_conditions,table:Isserlis_order_conditions} are satisfied for $p_1,p_2\in\{1,\dots,m\}$, that is, for each forest, the Runge-Kutta coefficient is equal to the Itô (respectively Stratonovich) coefficient.
Then, the method is of global weak order two for solving \eqref{equation:SDE_Ito} (respectively \eqref{equation:SDE_Strato}).
\end{theorem}

\begin{longtable}{C|C|C|C|C}
\text{Ex.\ts forest}&\text{Differential}&\text{RK coefficient } (p_1\neq p_2)&\text{Itô}&\text{Str.}\\\hline
{\forestA} & \phi_i f^{0,i} & \E[z^{0}_{i}] & 1 & 1\\
{\forestB} & \phi_{ij} f^{p_1,i} f^{p_1,j} & \E[z^{p_1}_{i} z^{p_1}_{j}] & 1 & 1\\
{\forestC} & \phi_i f^{p_1,i}_{i_1} f^{p_1,i_1} & \E[z^{p_1}_{i} Z^{p_1,p_1}_{i,i_1}] & 0 & \frac{1}{2} \\
\hline
{\forestD} & \phi_i f^{0,i}_{i_1} f^{0,i_1} & \E[z^{0}_{i}Z^{0,0}_{i,i_1}] & \frac{1}{2} & \frac{1}{2}\\
{\forestE} & \phi_{ij} f^{0,i} f^{0,j} & \E[z^{0}_{i}z^{0}_{j}] & 1 & 1\\
{\forestF} & \phi_i f^{0,i}_{i_1} f^{p_1,i_1}_{i_2} f^{p_1,i_2} & \E[z^{0}_{i}Z^{0,p_1}_{i,i_1}Z^{p_1,p_1}_{i_1,i_2}] & 0 & \frac{1}{4}\\
{\forestG} & \phi_i f^{p_1,i}_{i_1} f^{0,i_1}_{i_2} f^{p_1,i_2} & \E[z^{p_1}_{i}Z^{p_1,0}_{i,i_1}Z^{0,p_1}_{i_1,i_2}] & 0 & 0\\
{\forestH} & \phi_i f^{p_1,i}_{i_1} f^{p_1,i_1}_{i_2} f^{0,i_2} & \E[z^{p_1}_{i}Z^{p_1,p_1}_{i,i_1}Z^{p_1,0}_{i_1,i_2}] & 0 & \frac{1}{4}\\
{\forestI} & \phi_i f^{0,i}_{i_1 i_2} f^{p_1,i_1} f^{p_1,i_2} & \E[z^{0}_{i}Z^{0,p_1}_{i,i_1}Z^{0,p_1}_{i,i_2}] & \frac{1}{2} & \frac{1}{2}\\
{\forestJ} & \phi_i f^{p_1,i}_{i_1 i_2} f^{0,i_1} f^{p_1,i_2} & \E[z^{p_1}_{i}Z^{p_1,0}_{i,i_1}Z^{p_1,p_1}_{i,i_2}] & 0 & \frac{1}{4}\\
{\forestK} & \phi_{ij} f^{0,i}_{i_1} f^{p_1,i_1} f^{p_1,j} & \E[z^{0}_{i}Z^{0,p_1}_{i,i_1}z^{p_1}_{j}] & \frac{1}{2} & \frac{1}{2}\\
{\forestL} & \phi_{ij} f^{p_1,i}_{i_1} f^{0,i_1} f^{p_1,j} & \E[z^{p_1}_{i}Z^{p_1,0}_{i,i_1}z^{p_1}_{j}] & \frac{1}{2} & \frac{1}{2}\\
{\forestM} & \phi_{ij} f^{p_1,i}_{i_1} f^{p_1,i_1} f^{0,j} & \E[z^{p_1}_{i}Z^{p_1,p_1}_{i,i_1}z^{0}_{j}] & 0 & \frac{1}{2}\\
{\forestN} & \phi_{ijk} f^{0,i} f^{p_1,j} f^{p_1,k} & \E[z^{0}_{i}z^{p_1}_{j}z^{p_1}_{k}] & 1 & 1\\
{\forestO} & \phi_{i} f^{p_1,i}_{i_1} f^{p_1,i_1}_{i_2} f^{p_2,i_2}_{i_3} f^{p_2,i_3} & \E[z^{p_1}_{i}Z^{p_1,p_1}_{i,i_1}Z^{p_1,p_2}_{i_1,i_2}Z^{p_2,p_2}_{i_2,i_3}] & 0 & \frac{1}{8}\\
{\forestP} & \phi_{i} f^{p_1,i}_{i_1} f^{p_2,i_1}_{i_2} f^{p_1,i_2}_{i_3} f^{p_2,i_3} & \E[z^{p_1}_{i}Z^{p_1,p_2}_{i,i_1}Z^{p_2,p_1}_{i_1,i_2}Z^{p_1,p_2}_{i_2,i_3}] & 0 & 0\\
{\forestQ} & \phi_{i} f^{p_1,i}_{i_1} f^{p_2,i_1}_{i_2} f^{p_2,i_2}_{i_3} f^{p_1,i_3} & \E[z^{p_1}_{i}Z^{p_1,p_2}_{i,i_1}Z^{p_2,p_2}_{i_1,i_2}Z^{p_2,p_1}_{i_2,i_3}] & 0 & 0\\
{\forestR} & \phi_{i} f^{p_1,i}_{i_1} f^{p_1,i_1}_{i_2 i_3} f^{p_2,i_2} f^{p_2,i_3} & \E[z^{p_1}_{i}Z^{p_1,p_1}_{i,i_1}Z^{p_1,p_2}_{i_1,i_2}Z^{p_1,p_2}_{i_1,i_3}] & 0 & \frac{1}{4}\\
{\forestS} & \phi_{i} f^{p_1,i}_{i_1} f^{p_2,i_1}_{i_2 i_3} f^{p_1,i_2} f^{p_2,i_3} & \E[z^{p_1}_{i}Z^{p_1,p_2}_{i,i_1}Z^{p_2,p_1}_{i_1,i_2}Z^{p_2,p_2}_{i_1,i_3}] & 0 & 0\\
{\forestT} & \phi_{i} f^{p_1,i}_{i_1 i_2} f^{p_1,i_1} f^{p_2,i_2}_{i_3} f^{p_2,i_3} & \E[z^{p_1}_{i}Z^{p_1,p_1}_{i,i_1}Z^{p_1,p_2}_{i,i_2}Z^{p_2,p_2}_{i_2,i_3}] & 0 & \frac{1}{8}\\
{\forestU} & \phi_{i} f^{p_1,i}_{i_1 i_2} f^{p_2,i_1} f^{p_1,i_2}_{i_3} f^{p_2,i_3} & \E[z^{p_1}_{i}Z^{p_1,p_2}_{i,i_1}Z^{p_1,p_1}_{i,i_2}Z^{p_1,p_2}_{i_2,i_3}] & 0 & \frac{1}{4}\\
{\forestV} & \phi_{i} f^{p_1,i}_{i_1 i_2} f^{p_2,i_1} f^{p_2,i_2}_{i_3} f^{p_1,i_3} & \E[z^{p_1}_{i}Z^{p_1,p_2}_{i,i_1}Z^{p_1,p_2}_{i,i_2}Z^{p_2,p_1}_{i_2,i_3}] & 0 & 0\\
{\forestW} & \phi_{i} f^{p_1,i}_{i_1 i_2 i_3} f^{p_1,i_1} f^{p_2,i_2} f^{p_2,i_3} & \E[z^{p_1}_{i}Z^{p_1,p_1}_{i,i_1}Z^{p_1,p_2}_{i,i_2}Z^{p_1,p_2}_{i,i_3}] & 0 & \frac{1}{4}\\
{\forestX} & \phi_{ij} f^{p_1,i}_{i_1} f^{p_2,i_1}_{i_2} f^{p_2,i_2} f^{p_1,j} & \E[z^{p_1}_{i}Z^{p_1,p_2}_{i,i_1}Z^{p_2,p_2}_{i_1,i_2}z^{p_1}_{j}] & 0 & \frac{1}{4}\\
{\forestY} & \phi_{ij} f^{p_2,i}_{i_1} f^{p_1,i_1}_{i_2} f^{p_2,i_2} f^{p_1,j} & \E[z^{p_2}_{i}Z^{p_2,p_1}_{i,i_1}Z^{p_1,p_2}_{i_1,i_2}z^{p_1}_{j}] & 0 & 0\\
{\forestAB} & \phi_{ij} f^{p_2,i}_{i_1} f^{p_2,i_1}_{i_2} f^{p_1,i_2} f^{p_1,j} & \E[z^{p_2}_{i}Z^{p_2,p_2}_{i,i_1}Z^{p_2,p_1}_{i_1,i_2}z^{p_1}_{j}] & 0 & \frac{1}{4}\\
{\forestBB} & \phi_{ij} f^{p_1,i}_{i_1 i_2} f^{p_2,i_1} f^{p_2,i_2} f^{p_1,j} & \E[z^{p_1}_{i}Z^{p_1,p_2}_{i,i_1}Z^{p_1,p_2}_{i,i_2}z^{p_1}_{j}] & \frac{1}{2} & \frac{1}{2}\\
{\forestCB} & \phi_{ij} f^{p_2,i}_{i_1 i_2} f^{p_1,i_1} f^{p_2,i_2} f^{p_1,j} & \E[z^{p_2}_{i}Z^{p_2,p_1}_{i,i_1}Z^{p_2,p_2}_{i,i_2}z^{p_1}_{j}] & 0 & \frac{1}{4}\\
{\forestDB} & \phi_{ij} f^{p_1,i}_{i_1} f^{p_1,i_1} f^{p_2,j}_{j_1} f^{p_2,j_1} & \E[z^{p_1}_{i}Z^{p_1,p_1}_{i,i_1}z^{p_2}_{j}Z^{p_2,p_2}_{j,j_1}] & 0 & \frac{1}{4}\\
{\forestEB} & \phi_{ij} f^{p_1,i}_{i_1} f^{p_2,i_1} f^{p_1,j}_{j_1} f^{p_2,j_1} & \E[z^{p_1}_{i}Z^{p_1,p_2}_{i,i_1}z^{p_1}_{j}Z^{p_1,p_2}_{j,j_1}] & \frac{1}{2} & \frac{1}{2}\\
{\forestFB} & \phi_{ij} f^{p_1,i}_{i_1} f^{p_2,i_1} f^{p_2,j}_{j_1} f^{p_1,j_1} & \E[z^{p_1}_{i}Z^{p_1,p_2}_{i,i_1}z^{p_2}_{j}Z^{p_2,p_1}_{j,j_1}] & 0 & 0\\
{\forestGB} & \phi_{ijk} f^{p_1,i}_{i_1} f^{p_1,i_1} f^{p_2,j} f^{p_2,k} & \E[z^{p_1}_{i}Z^{p_1,p_1}_{i,i_1}z^{p_2}_{j}z^{p_2}_{k}] & 0 & \frac{1}{2}\\
{\forestHB} & \phi_{ijk} f^{p_1,i}_{i_1} f^{p_2,i_1} f^{p_1,j} f^{p_2,k} & \E[z^{p_1}_{i}Z^{p_1,p_2}_{i,i_1}z^{p_1}_{j}z^{p_2}_{k}] & \frac{1}{2} & \frac{1}{2}\\
{\forestIB} & \phi_{ijkl} f^{p_1,i} f^{p_1,j} f^{p_2,k} f^{p_2,l} & \E[z^{p_1}_{i}z^{p_1}_{j}z^{p_2}_{k}z^{p_2}_{l}] & 1 & 1
\\
\caption{Exotic order conditions of stochastic Runge-Kutta method of the form \eqref{equation:def_RKsto} up to weak order 2. The order conditions do not depend on the dimension of the problem and on the number of noise terms. The sums on all involved indices except $p_1$, $p_2$ are omitted for conciseness.}
\label{table:exotic_order_conditions}
\end{longtable}

\begin{longtable}{C|C|C|C|C}
\text{Dec.\ts forest}&\text{Differential}&\text{RK coefficient}&\text{Itô}&\text{Str.}\\\hline
{\forestJB} & \phi_{i} f^{p_1,i}_{i_1} f^{p_1,i_1}_{i_2} f^{p_1,i_2}_{i_3} f^{p_1,i_3} & \E[z^{p_1}_{i}Z^{p_1,p_1}_{i,i_1}Z^{p_1,p_1}_{i_1,i_2}Z^{p_1,p_1}_{i_2,i_3}] & 0 & \frac{1}{8}\\
{\forestKB} & \phi_{i} f^{p_1,i}_{i_1} f^{p_1,i_1}_{i_2 i_3} f^{p_1,i_2} f^{p_1,i_3} & \E[z^{p_1}_{i}Z^{p_1,p_1}_{i,i_1}Z^{p_1,p_1}_{i_1,i_2}Z^{p_1,p_1}_{i_1,i_3}] & 0 & \frac{1}{4}\\
{\forestLB} & \phi_{i} f^{p_1,i}_{i_1 i_2} f^{p_1,i_1} f^{p_1,i_2}_{i_3} f^{p_1,i_3} & \E[z^{p_1}_{i}Z^{p_1,p_1}_{i,i_1}Z^{p_1,p_1}_{i,i_2}Z^{p_1,p_1}_{i_2,i_3}] & 0 & \frac{3}{8}\\
{\forestMB} & \phi_{i} f^{p_1,i}_{i_1 i_2 i_3} f^{p_1,i_1} f^{p_1,i_2} f^{p_1,i_3} & \E[z^{p_1}_{i}Z^{p_1,p_1}_{i,i_1}Z^{p_1,p_1}_{i,i_2}Z^{p_1,p_1}_{i,i_3}] & 0 & \frac{3}{4}\\
{\forestNB} & \phi_{ij} f^{p_1,i}_{i_1} f^{p_1,i_1}_{i_2} f^{p_1,i_2} f^{p_1,j} & \E[z^{p_1}_{i}Z^{p_1,p_1}_{i,i_1}Z^{p_1,p_1}_{i_1,i_2}z^{p_1}_{j}] & 0 & \frac{1}{2}\\
{\forestOB} & \phi_{ij} f^{p_1,i}_{i_1 i_2} f^{p_1,i_1} f^{p_1,i_2} f^{p_1,j} & \E[z^{p_1}_{i}Z^{p_1,p_1}_{i,i_1}Z^{p_1,p_1}_{i,i_2}z^{p_1}_{j}] & \frac{1}{2} & 1\\
{\forestPB} & \phi_{ij} f^{p_1,i}_{i_1} f^{p_1,i_1} f^{p_1,j}_{j_1} f^{p_1,j_1} & \E[z^{p_1}_{i}Z^{p_1,p_1}_{i,i_1}z^{p_1}_{j}Z^{p_1,p_1}_{j,j_1}] & \frac{1}{2} & \frac{3}{4}\\
{\forestQB} & \phi_{ijk} f^{p_1,i}_{i_1} f^{p_1,i_1} f^{p_1,j} f^{p_1,k} & \E[z^{p_1}_{i}Z^{p_1,p_1}_{i,i_1}z^{p_1}_{j}z^{p_1}_{k}] & 1 & \frac{3}{2}\\
{\forestRB} & \phi_{ijkl} f^{p_1,i} f^{p_1,j} f^{p_1,k} f^{p_1,l} & \E[z^{p_1}_{i}z^{p_1}_{j}z^{p_1}_{k}z^{p_1}_{l}] & 3 & 3
\\
\caption{Isserlis order conditions of stochastic Runge-Kutta method of the form \eqref{equation:def_RKsto} for weak order 2. The order conditions do not depend on the dimension of the problem and on the number of noise terms. The sums on all involved indices except $p_1$ are omitted for conciseness.}
\label{table:Isserlis_order_conditions}
\end{longtable}

The weak order conditions are traditionally derived by use of Butcher series as explained in  \cite{Rossler06rta, Rossler06rkm, Debrabant08bsa}. The order conditions presented in \cref{table:exotic_order_conditions,table:Isserlis_order_conditions} are based on the more recent formalism of exotic forests \cite{Laurent20eab}. This formalism reduces the complexity in the sense that all redundant order conditions are removed, and terms for which the pair of colors (representing the different Brownian motions) appears several times are given as separate forests. Thus each forest represents exactly one unique order condition.
Note that the invariance in law of the moments of the Runge-Kutta coefficients when applying permutations, as prescribed in Assumption \ref{ass:RV_conditions}, is key to the representation of flows with tree structures.
We emphasize that the algebraic and combinatorial details underlying the computation of the order conditions can be skipped by the reader interested only in the creation of numerical methods of high weak order.
The proof of \cref{theorem:RK_conditions_general} and the study of the algebraic structures underlying stochastic numerics are postponed to Section \ref{section:general_conditions}.

\begin{remark}
The previous stochastic B-series formalisms \cite{Rossler06rta} typically rely on forests of the form \forestSB, where $i,j,k,l$ range from $1$ to $m$. The random variable corresponding to this forest will satisfy
\[ \frac{1}{h^2}\E[\int_0^h{W^j}\star dW^i \int_0^h dW^k \int_0^h dW^l] =
  \begin{cases}
    \text{Itô} / \text{Strat.} \\
    0 \; /\; \frac12 & \text{ if } i=j, \; k=l \text{ and } i \not=l, \\
    \frac12 \; /\; \frac12 & \text{ if } i=l, \;j=k \text{ and } i \not=k, \\
    1 \; /\; \frac32 & \text{ if } i=j=k=l, \\
    0 \; /\; 0 & \text{ otherwise. }
  \end{cases}
\]
Thus, the $m^4$ forests of the form \forestTB generate three non-trivial order conditions under \cref{ass:RV_conditions}. Using the exotic forests formalism, these
will appear as three different decorated forests, given by
\[ \pi_{d_1} = \forestUB, \qquad \pi_{d_2} = \forestVB, \qquad \pi_{d_ 3} = \forestWB. \]
The weak order condition corresponding to $\pi_{d_1}$ is then
\[ \E[\sum_{i,j,k,l} z_i^{p_1}Z_{i,j}^{p_1,p_1}z^{p_2}_kz^{p_2}_l] =
  \begin{cases}
    0 & \text{Itô} \\
    \frac12 & \text{Stratonovich}
  \end{cases}
\]
The other two conditions can be found in \cref{table:exotic_order_conditions,table:Isserlis_order_conditions}.
\end{remark}

We say that order conditions of the form $\sum \E[\dots]=0$ are \emph{potentially superfluous} as they can potentially be satisfied automatically with a particular choice of random coefficients. Note that there is a relatively high number of these order conditions in the Itô case.
In the next section, we propose a carefully chosen ansatz for the Runge-Kutta coefficients so that most potentially superfluous order conditions are satisfied automatically.

\subsection{Reduced order conditions for Itô SDEs}
\label{section:reduced_order_ito}
We consider coefficients of the form
\begin{equation}
\label{eq:ItoCoeff}
\begin{tabular}{lll}
 $z^{0}=\alpha \theta_0,\quad$ & $z^{p}=\beta\theta_p,$ \\
 $Z^{0,0}=A^0 \Theta_{0,0},\quad$ & $Z^{0,q}=B^0 \Theta_{0,q},$ \\
 $Z^{p,0}=A^1 \Theta_{p,0},\quad$ & $Z^{p,q}=B^1 \Theta_{p,q}.$
\end{tabular}
\end{equation}
where $\alpha\in \R^{s_1}$, $\beta\in\R^{s_2}$, $A^0\in \R^{s_1,s_1}$, $A^1\in \R^{s_2,s_1}$, $B^0\in \R^{s_1,s_2}$, $B^1\in \R^{s_2,s_2}$ and $(s_1,s_2)$ is the number of stages for the deterministic and stochastic parts of the method.
This choice of coefficients fits in the class of methods \eqref{equation:def_RKsto} by choosing $s=\max(s_1,s_2)$ and filling the missing entries of the Runge-Kutta coefficients by zeros.
The standard methodology mimicking strong expansions and using weak approximations of iterated stochastic integrals is tedious and yields unnecessarily complicated methods. We simplify this approach by instead choosing random variables such that most potentially superfluous order conditions are satisfied automatically. Let $\eta_p$, $p=0,\dots,m$,  and $\theta_p$, $p=1,\dots,m$, be discrete independent random variables satisfying
$$\P(\eta_p=\pm 1)=\frac{1}{2}, \quad
\P(\theta_p=\pm\sqrt{2+\sqrt{3}})=\frac{3-\sqrt{3}}{12},\quad \P(\theta_p=\pm\sqrt{2-\sqrt{3}})=\frac{3+\sqrt{3}}{12}.$$
Define for $c\in(0,\frac12)$
\begin{equation}
\label{eq:RVIto}
\begin{tabular}{lll}
 $\theta_0=1,$ & $\Theta_{0,p}=\theta_p+\eta_p\sqrt{\frac1{2c}-1}, \quad$ & $\Theta_{p,q}=\theta_q(1+\eta_0), \quad q> p\geq 1,$ \\
 $\Theta_{0,0}=1,$ & $\Theta_{p,0}=1-\eta_p\theta_p\sqrt{\frac{2c}{1-2c}}, \quad$ & $\Theta_{p,q}=\theta_q(1-\eta_0), \quad 1\leq q< p$, \\
 $\Theta_{p,p}=-3\theta_p+\theta_p^3$.
\end{tabular}
\end{equation}
These random variables do, for $c$ not too close to $\{0,\frac12\}$, not impact stability more than using Gaussians or iterated stochastic integrals as their moments are of similar (if not smaller) amplitude.

Using the ansatz \eqref{eq:ItoCoeff} and the random variables \eqref{eq:RVIto} in \cref{theorem:RK_conditions_general}, \cref{ass:RV_conditions} is automatically satisfied and we obtain the following result, where most potentially superfluous order conditions vanish,
leaving us with 9 reduced order conditions for weak order two. By choosing carefully the random variables, we got rid of most of the 59 order conditions of \cite{Debrabant09foe} and improved further on the 15 order conditions of \cite{Tang17ews}.
\begin{theorem}
\label{theorem:reduced_RK_conditions_Ito}
The reduced order conditions for weak order one for solving the Itô SDE \eqref{equation:SDE_Ito} by method \eqref{equation:def_RKsto} with method coefficients \eqref{eq:ItoCoeff}, random variables \eqref{eq:RVIto} and $c\in(0,\frac12)$ are:
\begin{multicols}{3}
    \begin{enumerate}
\item $\alpha^\top \ind=1$,\label[cond]{cond:1}
\item $\beta^\top \ind=1$,\label[cond]{cond:2}
    \end{enumerate}
\end{multicols}
\noindent and the additional conditions for weak order two are:
\begin{multicols}{3}
	\begin{enumerate}
\setcounter{enumi}{2}
\item $\alpha^\top A^0 \ind =\frac{1}{2}$,\label[cond]{cond:3}
\item $\alpha^\top B^0 \ind =\frac{1}{2}$,\label[cond]{cond:4}
\item $\alpha^\top (B^0 \ind)^{\diam 2}=c$,\label[cond]{cond:5}
\item $\beta^\top A^1 \ind =\frac{1}{2}$,\label[cond]{cond:6}
\item $\beta^\top B^1 \ind =\frac{1}{2}$,\label[cond]{cond:7}
\item $\beta^\top (B^1 \ind)^{\diam 2}=\frac{1}{4}$,\label[cond]{cond:8}
\item $\beta^\top B^1 B^1\ind =0$\label[cond]{cond:9},
	\end{enumerate}
\end{multicols}
where we use the Hadamard product on vectors in~$\R^s$:
$$y^{1}\diam \dots \diam y^{n}=\bigg(\prod_{k=1}^n y^{k}_i\bigg)_{i=1,\dots,s}.$$
It is also possible to choose $c=\frac12$ when adapting \eqref{eq:RVIto} by $\Theta_{0,p}=\theta_p$, $\Theta_{p,0}=1$, in which case also the following additional order condition needs to be fulfilled for weak order two:
	\begin{enumerate}
\setcounter{enumi}{9}
\item $\beta^\top A^1B^0\ind=0$.\label[cond]{cond:10}
	\end{enumerate}
\end{theorem}

Note that \cref{cond:9} and, in case  $c=\frac12$, also \cref{cond:10}, in Theorem \eqref{theorem:reduced_RK_conditions_Ito}, are still derived from potentially superfluous conditions.
Note also that for $c<\frac12$, we always can choose $A^0=B^0$ and $A^1=B^1$, and that the parameters $\alpha$, $A^0$ and $B^0$ on the one hand and $\beta$, $A^1$ and $B^1$ are independent of each other. A natural choice for $c$ is $c=\frac{1}{4}$, which ensures that \cref{cond:1,cond:3,cond:4,cond:5} and \cref{cond:2,cond:6,cond:7,cond:8} are congruent and we thus can choose $B^0=B^1$ and $\alpha=\beta$. Another natural choice would be $c=\frac13$, which allows for a method with $A^0=B^0$ already fulfilling one of the deterministic order three conditions.
Choosing $c=\frac12$ on the other hand reduces the number of random variables needed per step from $2m+1$ to $m+1$.

\begin{remark}The coefficients \eqref{eq:ItoCoeff} do not allow for weak order three methods, as
the following second and third order conditions are contradictive:
  \begin{align*}
    \forestXB && (\beta^\top \ind)^2 \E [\theta_p^2] = 1 &&
    \Rightarrow &&& \beta^\top \ind \neq0 \text{ and } \E[\theta_p^2] \neq 0, \\
    \forestYB && (\beta^\top \ind)(\beta^\top B^1B^1\ind)
    \E[\theta_p^2\Theta_{pp}^2] = 0 && \Rightarrow
    &&& \beta^\top \ind=0 \text{ or } \beta^\top B^1 B^1\ind = 0 \\
     &&  &&   &&& \text{ or }  \E[\theta_p^2\Theta_{pp}^2] = 0, \\
    \forestAC && (\beta^\top B^1 \ind )^2 \E[\theta_p^2\Theta_{pp}^2] = \frac{1}{2}
    && \Rightarrow &&&
    \beta^\top B^1 \ind  \neq 0 \text{ and } \E[\theta_p^2\Theta_{pp}^2] \neq 0, \\
    \forestBC &&
    (\beta^\top B^1 B^1\ind)^2  \E[\theta_p^2\Theta_{pp}^4]=\frac16
    && \Rightarrow &&& \beta^\top B^1 B^1\ind \neq 0 \text{ and }
    \E[\theta_p^2\Theta_{pp}^4] \neq 0.
  \end{align*}
 Hence a more general ansatz for third order Runge-Kutta coefficients is needed.
\end{remark}

\subsection{Reduced order conditions for Stratonovich SDEs}

There are fewer potentially superfluous order conditions in the Stratonovich case, which prevents us from removing most of the conditions with a certain choice of random variables. We still present a substantial reduction of order conditions compared to the literature.
We consider coefficients of the form
\begin{equation}\label{eq:StratoCoeff}
\begin{tabular}{lll}
 $z^{0}=\alpha \theta_0,\quad$ & $z^{p}=\beta\theta_p,$ \\
 $Z^{0,0}= A^0 \Theta_{0,0},\quad$ & $Z^{0,q}=B^0 \Theta_{0,q},$ \\
 $Z^{p,0}= A^1 \Theta_{p,0},\quad$ & $Z^{p,q}=B^1 \Theta_{p,q}\ind_{p\neq q}+\hat{B}^1 \Theta_{p,p}\ind_{p=q}.$
\end{tabular}
\end{equation}
Let $\eta_p$, $p=0,\dots,m$, and $\theta_p$, $p=1,\dots,m$, be discrete independent random variables satisfying
$$\P(\eta_p=\pm 1)=\frac{1}{2}, \quad
\P(\theta_p=\pm\sqrt{3})=\frac{1}{6},\quad \P(\theta_p=0)=\frac{2}{3}.$$
Define
\begin{equation}
\label{eq:RVStrato}
\begin{tabular}{lll}
 $\theta_0=1,$ & $\Theta_{0,p}=\theta_p+\eta_p\sqrt{\frac1{2c}-1}, \quad$ & $\Theta_{p,q}=\theta_q(1+\eta_0), \quad q> p\geq 1,$ \\
 $\Theta_{0,0}=1, \quad$ & $\Theta_{p,0}=1-\eta_p\theta_p\sqrt{\frac{2c}{1-2c}}, \quad$ & $\Theta_{p,q}=\theta_q(1-\eta_0), \quad 1\leq q< p.$\\
 $\Theta_{p,p}=\theta_p, \quad$ &  &
\end{tabular}
\end{equation}
Using the ansatz \eqref{eq:StratoCoeff} and the random variables \eqref{eq:RVStrato} in \cref{theorem:RK_conditions_general}, \cref{ass:RV_conditions} is automatically satisfied and we obtain the following 26 reduced order conditions. These conditions improve significantly on the 55 order conditions in \cite{Rossler07sor}.
\begin{theorem}
\label{theorem:reduced_RK_conditions_Strato}
The reduced order conditions for weak order one for solving the Stratonovich SDE \eqref{equation:SDE_Strato} by method \eqref{equation:def_RKsto} with method coefficients \eqref{eq:StratoCoeff}, random variables \eqref{eq:RVStrato} and $c\in(0,\frac12)$ are:
\begin{multicols}{3}
	\begin{enumerate}
\item $\alpha^\top \ind=1$,\label[cond]{cond:Strato1}
\item $\beta^\top \ind=1$,
\item $\beta^\top \hat{B}^1 \ind =\frac{1}{2}$,
	\end{enumerate}
\end{multicols}
\noindent and the additional conditions for weak order two are:
\begin{multicols}{3}
	\begin{enumerate}[leftmargin=*]
\setcounter{enumi}{3}
\item $\alpha^\top A^0 \ind =\frac{1}{2}$,
\item $\alpha^\top B^0 \ind =\frac{1}{2}$,\label[cond]{cond:Strato5}
\item $\alpha^\top (B^0 \ind)^{\diam 2} =c$,\label[cond]{cond:Strato6}
\item $\alpha^\top B^0 \hat{B}^1 \ind =\frac{1}{4}$,
\item $\beta^\top A^1 \ind =\frac{1}{2}$,
\item $\beta^\top( \hat{B}^1 \ind \diam A^1 \ind) =\frac{1}{4}$,
\item $\beta^\top \hat{B}^1 A^1 \ind =\frac{1}{4}$,
\item $\beta^\top B^1 \ind =\frac{1}{2}$,
\item $\beta^\top (B^1 \ind)^{\diam 2} =\frac{1}{4}$,
\item $\beta^\top \hat{B}^1 B^1 \hat{B}^1\ind =\frac{1}{8}$,
\item $\beta^\top( \hat{B}^1\ind \diam B^1 \hat{B}^1\ind) =\frac{1}{8}$,
\item $\beta^\top (\hat{B}^1 \ind \diam (B^1 \ind)^{\diam 2}) =\frac{1}{8}$,
\item $\beta^\top( B^1\ind \diam \hat{B}^1 B^1\ind) =\frac{1}{8}$,
\item $\beta^\top \hat{B}^1 (B^1 \ind)^{\diam 2} =\frac{1}{8}$,
\item $\beta^\top( B^1 \ind \diam \hat{B}^1 \ind) =\frac{1}{4}$,
\item $\beta^\top \hat{B}^1 B^1\ind =\frac{1}{4}$,
\item $\beta^\top B^1 \hat{B}^1\ind =\frac{1}{4}$,
\item $\beta^\top( \hat{B}^1\ind \diam \hat{B}^1 \hat{B}^1\ind) =\frac{1}{8}$,
\item $\beta^\top \hat{B}^1 (\hat{B}^1 \ind)^{\diam 2} =\frac{1}{12}$,
\item $\beta^\top \hat{B}^1 \hat{B}^1 \hat{B}^1\ind =\frac{1}{24}$,
\item $\beta^\top (\hat{B}^1\ind)^{\diam 3} =\frac{1}{4}$,
\item $\beta^\top (\hat{B}^1 \ind)^{\diam 2} =\frac{1}{3}$,
\item $\beta^\top \hat{B}^1 \hat{B}^1\ind =\frac{1}{6}$.
	\end{enumerate}
\end{multicols}
It is also possible to choose $c=\frac12$ when adapting \eqref{eq:RVStrato} by $\Theta_{0,p}=\theta_p$, $\Theta_{p,0}=1$, in which case also the following additional order condition needs to be fulfilled for weak order two:
	\begin{enumerate}
\setcounter{enumi}{26}
\item $\beta^\top A^1B^0\ind=0$.\label[cond]{cond:27}
	\end{enumerate}
\end{theorem}
Analogously to the Itô case, one can reduce the number of stochastic variables needed per step from $2m+1$ to $m+1$ by adding \cref{cond:27}.

\section{New second order stochastic Runge-Kutta methods}
\label{section:new_methods}

We use the reduced order conditions of Section \ref{section:reduced_order_conditions} to propose a handful of new stochastic methods with the minimal number of function evaluations for solving \eqref{equation:SDE_Ito}-\eqref{equation:SDE_Strato} with second weak order.
We focus on explicit methods and methods with high deterministic order. Further IMEX methods are presented in Appendix \ref{section:IMEX}.

\subsection{Optimal stochastic Runge-Kutta methods}

Thanks to the reduced order conditions, we propose a variety of new simple stochastic Runge-Kutta methods of weak order two.
We emphasize that, similarly to the deterministic setting, an explicit method needs at least two evaluations of $f^0$ and two evaluations of each $f^p$ per step.
The Butcher tableaux are written in the following ways for Itô and Stratonovich.
\[
  \text{Itô:}\quad \begin{array}{c|c}
    A^0 & B^0 \\ \hline A^1 & B^1 \\ \hline \alpha^\top & \beta^\top
  \end{array}, \qquad
  \text{Stratonovich:}\quad \begin{array}{c|c|c}
    A^0 & B^0 \\ \hline A^1 & B^1 & \hat{B}^1 \\ \hline \alpha^\top & \beta^\top
  \end{array}
\]

\paragraph*{Itô explicit}
A $(2,2)$-stages method is the following mix of the Heun and explicit midpoint methods.
\[
 \begin{array}{rl}
X_{n+1}&=X_n+\frac{h}{2}f^0(X_n)+\frac{h}{2}f^0\Big(X_n+hf^0(X_n)\\
&+\sqrt{h}\sum_{q=1}^m\theta_{q}f^q(X_n)\Big)\\
&+\sqrt{h} \sum_{p=1}^m\theta_p f^p\Big(X_n+\frac{h}{2}f^0(X_n)
+\frac{\sqrt{h}}{2}\sum_{q=1}^m\Theta_{p,q}f^q(X_n)\Big)
  \end{array}
  \
  \begin{array}{cc|cc}
    0 & 0 & 0 & 0 \\
    1 & 0 & 1 & 0 \\ \hline
    0 & 0 & 0 & 0 \\
    \frac12 & 0 & \frac{1}{2} & 0 \\ \hline
    \frac12 & \frac12 & 0 & 1
  \end{array}
\]
This method, in the following denoted by \Method{AAK4}, is an improvement of the methods presented both in \cite{Debrabant09foe} and \cite{Tang17ews}, as it uses only $m+1$ random variables per step and attains order 2 for solving \eqref{equation:SDE_Ito} (with $c=\frac12$ in \cref{theorem:reduced_RK_conditions_Ito}), with the optimal number of stages.

\paragraph*{Itô implicit}
When constructing an $(1,2)$-stages method of order 2 for solving \eqref{equation:SDE_Ito}, we observe that \cref{cond:1} in \cref{theorem:reduced_RK_conditions_Ito} implies that $\alpha=1$, which together with \cref{cond:4,cond:5} implies that we must choose $c=\frac14$ in this case. A possible method is then the following:
\[
 \begin{array}{rl}
H_1^0&=X_n+\frac{h}{2}f^0(H_1^0)+\frac{\sqrt{h}}{2}\sum_{q=1}^m\Theta_{0,q}f^q(H_1^q)\\
H_1^p&=X_n+\sqrt{h}\sum_{q=1}^m\Theta_{p,q}f^q(H_1^q)\\
H_2^p&=X_n+\frac{h}{2}\Theta_{p,0}f^0(H_1^0)+\sqrt{h}\sum_{q=1}^m\Theta_{p,q}\Big(f^q(H_2^q)-\frac{1}{2} f^q(H_1^q)\Big)\\
X_{n+1}&=X_n+hf^0(H_1^0)+\sqrt{h}\sum_{p=1}^m\theta_p f^p(H_2^p)
  \end{array}
  \
  \begin{array}{c|cc}
    \frac12 & \frac12 & 0 \\ \hline
    0 & 1 & 0 \\
    \frac12 & -\frac12 & 1 \\ \hline
    1 & 0 & 1
  \end{array}
\]
An order 2 method with one stage for the stochastic part is impossible in general as the order conditions of \cref{theorem:RK_conditions_general} for the forests $\forestCC$ and $\forestDC$ cannot be satisfied simultaneously. If the noise is additive, it is however straightforwardly achieved \cite{Debrabant10rkm,Laurent20eab}.

\paragraph*{Stratonovich explicit}
 We propose the following $(2,4)$-stages explicit method with $c=\frac12$ of weak order two  for solving \eqref{equation:SDE_Strato}, which is a mild improvement on the methods in \cite{Rossler07sor} as it uses fewer stages in the deterministic part, and only $m+1$ random variables per step instead of $2m-1$.
 \[
   \begin{array}{cc|cccc|cccc}
     0 & 0 & 0 & 0 & 0 & 0 \\
     1 & 0 & 0 & 1 & 0 & 0 \\
     \hline
     0 & 0 & 0 & 0 & 0 & 0 & 0 & 0 & 0 & 0 \\
     \frac12 & 0 & \frac12 & 0 & 0 & 0 & \frac12 & 0 & 0 & 0 \\
     \frac12 & 0 & -1 & \frac32 & 0 & 0 & -\frac12 & \frac12 & 0 & 0 \\
     \frac12 & 0 & -1 & \frac32 & 0 & 0 & -\frac32 & \frac32 & 1 & 0 \\ \hline
     \frac12 & \frac12 & 0 & \frac23 & \frac16 & \frac16
   \end{array}
 \]
 \vskip-4ex
 The number of stochastic stages is minimal thanks to the condition for $\forestEC$ in \cref{theorem:RK_conditions_general}.
 
\paragraph*{Stratonovich implicit}
Analogously to the Itô case, it follows by \cref{cond:Strato1,cond:Strato5,cond:Strato6} in \cref{theorem:reduced_RK_conditions_Strato} that a $(1,2)$-stage method for solving \eqref{equation:SDE_Strato} with weak order two requires $c=\frac14$. An example for such a method is given below. Observe that both $A^0$ and $B^1$ correspond to well-known Gauss–Legendre methods.
A $(1,1)$-stages method cannot be achieved as the order conditions associated to $\forestFC$ and $\forestGC$ in \cref{theorem:RK_conditions_general} would bring a contradiction.
\[
  \begin{array}{c|cc|cc}
    \frac12  & \frac14 & \frac14 \\
    \hline
    \frac12 & \frac14 & \frac14 & \frac14 & \frac{3+2\sqrt{3}}{12} \\
    \frac12 & \frac14 & \frac14 & \frac{3-2\sqrt{3}}{12} & \frac14 \\ \hline
    1 & \frac12 & \frac12
  \end{array}
\]

\subsection{Stochastic Runge-Kutta methods with high deterministic order}

Let us now present a handful of explicit methods with second weak order, optimal number of function evaluations, and third deterministic order, as these methods often display improved error constants and rates of convergence for stochastic perturbation of deterministic dynamics.

\paragraph*{Itô explicit of deterministic order 3}
We propose the following two methods \Method{AAK5} and \Method{AAK3} (respectively with $c=\frac12$
and $c=\frac13$) of order 2 and deterministic order 3. Note that the deterministic part of these
methods coincides with Kutta's RK32 \cite{Butcher16nmf}.
\[
  \begin{array}{ccc|cc}
      0 & 0 & 0 & 0 & 0 \\
\frac12 & 0 & 0 & \frac35-\frac{\sqrt{6}}{10} & 0 \\
     -1 & 2 & 0 & \frac35+\frac25\sqrt{6} & 0\\ \hline
    0 & 0 & 0 & 0 & 0 \\
    \frac12 & 0 & 0 & \frac12 & 0 \\ \hline
    \frac16 & \frac23 & \frac16 & 0 & 1
  \end{array}
  \quad \quad \quad \quad \quad 
  \begin{array}{ccc|cc}
    0 & 0 & 0 & 0 & 0 \\
    \frac12 & 0 & 0 & \frac12 & 0 \\
    -1 & 2 & 0 & 1 & 0\\ \hline
    0 & 0 & 0 & 0 & 0 \\
    \frac12 & 0 & 0 & \frac12 & 0 \\ \hline
    \frac16 & \frac23 & \frac16 & 0 & 1
  \end{array}
\]

\paragraph*{Stratonovich explicit of deterministic order 3 ($c=\frac12$)}\mbox{}\par
\[
  \begin{array}{ccc|cccc|cccc}
    0 & 0 & 0 & \frac12 - \frac{\sqrt{3}}{2} & 0 & 0 & 0 \\
    \frac13 & 0 & 0 & \sqrt{3}-1 & 0 & 0 & 0 \\
    0 & \frac23 & 0 & \frac16 + \frac{\sqrt{3}}{6} & 0 & 0 & \frac13 \\
    \hline
    0 & 0 & 0 & 0 & 0 & 0 & 0 & 0 & 0 & 0 & 0 \\
    \frac12 & 0 & 0 & \frac12 & 0 & 0 & 0 & \frac12 & 0 & 0 & 0 \\
    -\frac12 & 1 & 0 & -1 & \frac32 & 0 & 0 & -\frac12 & \frac12 & 0 & 0 \\
    \frac12 & 0 & 0 & -\frac12 & \frac32 & -\frac12 & 0 & -\frac32 & \frac32 & 1 & 0 \\ \hline
    \frac14 & 0 & \frac34 & 0 & \frac23 & \frac16 & \frac16
  \end{array}
\]

\section{Numerical applications}
\label{section:numerical_experiments}

In this section, in \cref{ex:KP,ex:DR} we first numerically confirm the theoretical findings on two test problems and conclude then with \cref{eq:invmeasuresampling} illustrating an application to invariant measure sampling.

In \cref{ex:KP,ex:DR}, we will compare the computational effort and accuracy of the newly derived methods \Method{AAK4}, \Method{AAK5} and \Method{AAK3} with some integrators from the literature, namely \Method{DRI1} from \cite{Debrabant09foe}, \Method{W2Ito1} from \cite{Tang17ews}, \Method{RI5} and \Method{RI6} from \cite{Rossler09sor}. Note that \Method{RI6} coincides for $m=1$ with Platen's second order method \cite{Platen84zza}. 
The numbers $N_d$ and $N_s$ of evaluations of the integrands $f^0$ and $f^p$, $p=1,\dots,m$, as well as the number $N_r$ of random variables per step of the new methods are lower as presented in \cref{table:comparison_methods}. In the numerical illustrations, the solution $\E[\phi(X(t))]$ will be approximated with step sizes $2^{-1}, \ldots, 2^{-5}$. We will define the computational effort per time step as $N_d+mN_s+N_r$ \cite{Debrabant09foe} and the numerically observed order of convergence $\hat{p}$ as slope of the regression line in the double--logarithmic plot of the approximated weak error vs.\ step size. The expectation in the error term ${|\E}[\phi(X_N) - \phi(X(T_N))]|$ will be approximated by Monte Carlo simulation ${|\hat{\E}}[\phi(X_N) - \phi(X(T_N))]|$, using 40000 batches of 25000 simulations.
\begin{table}[htb!]
\begin{center}
\begin{tabular}{c|cccccc}
  Method&$p_D$&$p_S$&$N_d$&$N_s$&$N_r$ if $m=1$&$N_r$ for $m>1$\\\hline
  \Method{RI6}&2&2&2&5&1&$2m-1$\\
  \Method{AAK4}&2&2&2&2&1&$m+1$\\\hline
  \Method{RI5},\Method{DRI1}&3&2&3&5&1&$2m-1$\\
  \Method{W2Ito1}&3&2&3&3&2&$m+2$\\
  \Method{AAK3}
  &3&2&3&2&2&$2m+1$\\
  \Method{AAK5}&3&2&3&2&1&$m+1$
\end{tabular}
\end{center}
\caption{Comparison of the considered explicit stochastic Runge-Kutta methods for solving \eqref{equation:SDE_Ito} in terms of deterministic order $p_D$ and  stochastic order $p_S$, numbers $N_d$ and $N_s$ of function evaluations and number $N_r$ of random variables.}
\label{table:comparison_methods}
\end{table}
\tikzexternalenable
\begin{ex}\label{ex:KP}As first example, we consider the non-linear
SDE~\cite{Kloeden92nso,Mackevicius01sow,Debrabant09foe}
\begin{equation} \label{Simu:nonlinear-SDE2}
    dX(t) = \left( \tfrac{1}{2} X(t) + \sqrt{X(t)^2 + 1} \right) \,
    dt + \sqrt{X(t)^2 + 1} \, dW(t), \qquad X(0)=0,
\end{equation}
on the time interval $I=[0,2]$ with the solution $X(t) = \sinh (t +
W(t))$. Choosing $\phi(x)=p(\arsinh(x))$, where $p(z) = z^3 -
6z^2 + 8z$, the expectation of the solution is given by
\begin{equation}
    \E[\phi(X(t))] = t^3 - 3t^2 + 2t.
\end{equation}
The simulation results at time $t=2$, presented in \cref{fig:exKP22,fig:exKP32}, show that with one noise term, the new methods behave similarly in terms of convergence to the ones from the literature, with a slightly reduced cost. The method \Method{AAK3} displays higher convergence rate than expected.

\initcompeffort{1}
\numfigure{\ConvNl}{e2}{AAK4,Platen}{Numerical results for \cref{ex:KP}, methods of order (2,2)\label{fig:exKP22}}
\numfigure{\ConvNl}{e2}{AAK5,AAK3,DRI1,RI5,W2Ito1}{Numerical results for \cref{ex:KP}, methods of order (3,2)\label{fig:exKP32}}
\end{ex}

\begin{ex}\label{ex:DR}As second example, we consider a nonlinear SDE with 10 Wiener processes \cite{Debrabant09foe,Tang17ews}
\begin{eqnarray}\nonumber
\lefteqn{dX(t)=X(t)\, dt+
\frac1{10}\sqrt{X(t)+\frac12}\, dW_1(t)+
\frac1{15}\sqrt{X(t)+\frac14}\, dW_2(t)}\\
&&+\nonumber
\frac1{20}\sqrt{X(t)+\frac15}\, dW_3(t)+
\frac1{25}\sqrt{X(t)+\frac1{10}}\, dW_4(t)+
\frac1{40}\sqrt{X(t)+\frac1{20}}\, dW_5(t)
\\&&+\nonumber
\frac1{25}\sqrt{X(t)+\frac12}\, dW_6(t)+
\frac1{20}\sqrt{X(t)+\frac14}\, dW_7(t)+
\frac1{15}\sqrt{X(t)+\frac15}\, dW_8(t)\\
&&+\label{Simu:nonlinear-SDE3}
\frac1{20}\sqrt{X(t)+\frac1{10}}\, dW_9(t)+
\frac1{25}\sqrt{X(t)+\frac1{20}}\, dW_{10}(t),\qquad X(0)=1.
\end{eqnarray}
We approximate the fourth moment, i.\,e., $\phi(x)=x^4$, with exact solution
\begin{equation}
\E[\phi(X(t))]=\frac{\scriptstyle 4625768169}{\scriptstyle 73570420483600}-\frac{\scriptstyle 2998776077847}{\scriptstyle 113706563209000}e^{\frac{731453}{360000}t}
+\frac{\scriptstyle 80235120932849}{\scriptstyle 78178246418000}e^{\frac{251453}{60000}t}
\end{equation}
on the time interval $I=[0,1]$. The simulation results at time $t=1$ are shown in \cref{fig:exDR22,fig:exDR32}. We observe that with multiple noise terms, the new methods have similar convergence to the standard methods, but with a significantly reduced cost.
The method \Method{AAK3} displays an improved constant of convergence.


\initcompeffort{10}
\numfigure{\ConvNl}{e4}{AAK4,RI6}{Numerical results for \cref{ex:DR}, methods of order (2,2)\label{fig:exDR22}}
\numfigure{\ConvNl}{e4}{AAK5,AAK3,DRI1,RI5,W2Ito1}{Numerical results for \cref{ex:DR}, methods of order (3,2)\label{fig:exDR32}}

 \end{ex}
\tikzexternaldisable
\begin{ex}\label{eq:invmeasuresampling}
For our last example, we are interested in the sampling of the invariant measure of the following ergodic dynamics in a context of molecular dynamics:
\begin{equation}
\label{equation:mul_Langevin}
dX(t)=F(X(t))dt +\Div(D^2(X(t)))dt+\sqrt{2} D(X(t)) dW(t), \quad F(x)=-D^2(x)\nabla V(x)
\end{equation}
where $V\colon \R^d\rightarrow\R$ is a smooth potential, $W$ is a $d$-dimensional Brownian motion, and $D\colon\R^d\rightarrow\R^{d\times d}$ is smooth.
Choosing $D(x)=I_d$ yields the standard overdamped Langevin dynamics \cite{Lelievre10fec}.
The dynamics of equation \eqref{equation:mul_Langevin} is naturally ergodic under technical assumptions on $V$ and $D$, that is, for all test function $\phi$,
\[\lim_{T\rightarrow\infty}\frac{1}{T}\int_0^T \phi(X(t))dt=\int_{\R^d} \phi(x) \rho_\infty(x) dx,\quad \rho_\infty(x)\propto \exp(-V(x)).
\]
The choice of a non-constant $D$ map does not modify the sampled measure $\rho_\infty$, but it can allow for more efficient sampling.
In \cite{Bronasco25els}, a generalisation of the popular Leimkuhler-Matthews integrator \cite{Leimkuhler13rco} for sampling the invariant measure of \eqref{equation:mul_Langevin} with order two is proposed. It is a postprocessed integrator \cite{Vilmart15pif} and it relies on a non-detailed second order weak integrator for the Itô SDE:
\[dX(t)=\sqrt{2} D(X(t)) dW(t).\]
Using our analysis, we propose the following postprocessed method for sampling the invariant measure of \eqref{equation:mul_Langevin},
\begin{align*}
H_n&=X_n+\frac{h}{4} F(\overline{X}_{n-1}),\\
X_{n+1}&=X_n
	+h F(\overline{X}_n)
	+\sqrt{2h}\sum_p D_{:,p}(H_n+\sqrt{\frac{h}{2}}\sum_q D_{:,q}(H_n) \Theta_{p,q}^n)\theta^n_p,\\
\overline{X}_n&=X_n+\sqrt{\frac{h}{2}}\sum_p D_{:,p}(H_n)\theta^n_p,
\end{align*}
with the random variables defined in Section \ref{section:reduced_order_ito} and for any initial condition $\overline{X}_{-1}=X_0$.
Assuming ergodicity of the scheme, applying the analysis of \cite{Phillips25nwc} yields that $(\overline{X}_n)$ is an explicit method of first weak order and of second order for sampling the invariant measure of \eqref{equation:mul_Langevin}.
This method uses only one evaluation of $F$ and two evaluations of $D$ per step. It does not rely on a specific form of $D$ as in \cite{Phillips25nwc}.
\end{ex}

\section{High order analysis with decorated and exotic series}
\label{section:general_conditions}

We derive here the general order conditions for the weak approximation of stochastic dynamics. While the idea is not new \cite{Rossler06rta,Rossler06rkm}, we present a new approach with exotic forests that simplifies the analysis and reduces the number of forests, together with modern algebraic tools from Hopf algebra theory for the derivation of stochastic order conditions.

\subsection{Decorated and exotic Butcher forests}

We consider graphs $\pi=(V,E)$ where $V$ is a finite set of nodes and $E\subset V\times V$ is a set of directed edges. If $e = (v,w) \in E$, the edge $e$ is going from the node $v$ to the node $w$, $v$ is a predecessor of $w$, and $w$ is a successor of $v$. We impose that each node has at most one outgoing edge.
The nodes that do not have a successor are called roots and we impose that the connected components, called trees, have exactly one root. Such graphs are then called a forest. By convention, we draw the edges going from top to bottom and the roots as the bottommost nodes.
For instance, the following graph has three roots:
\[V=\{1,2,3,4,5,6\},\quad E=\{(4,2),(5,2),(6,3)\},\quad \pi=(V,E)=\forestHC.\]
We attach specific decorations, also called colours in numerics, to the graphs. A decoration of a graph $\pi=(V,E)$ is a map $d\colon V\rightarrow\N$ such that for $n>0$, $\abs{d^{-1}(n)}\in 2\N$. A decorated graph is written $\pi_d$.
Given two decorations $d_1$ and $d_2$ of a given graph $\pi$, we say that $d_2$ is finer than $d_1$, written $d_2 \leq d_1$, if there exists $\alpha\colon \N\rightarrow\N$ such that $d_1=\alpha \circ d_2$ and $\alpha(n)=0$ if and only if $n=0$. The map $\alpha$ identifies colours. If $d_1\leq d_2$ and $d_2\leq d_1$, the decorations are equivalent. We write $d_1< d_2$ if $d_1\leq d_2$ and $d_1$ and $d_2$ are not equivalent.
The nodes of decoration $0$ are drawn in black.

\begin{ex}\label{ex_decorations}
A forest $\pi$ can be decorated in different ways:
\begin{equation}
\label{equation:ex_decorations}
\pi_{d_1}=\forestIC, \
\pi_{d_2}=\forestJC, \
\pi_{d_3}=\forestKC, \
\pi_{d_4}=\forestLC, \
\pi_{d_5}=\forestMC, \
\pi_{d_6}=\forestNC.
\end{equation}
In this example, the decorations $d_3$, $d_4$, $d_5$, $d_6$ are finer than $d_2$.
\end{ex}

The numerical and the exact flow of stochastic differential equations rewrite naturally with decorated forests. However, let us introduce a simpler set of forests that is sufficient for expanding the exact flow, in the spirit of \cite{Laurent20eab, Bronasco22ebs}.
\begin{definition}
Two decorated graphs $\pi^1_{d_1}$ and $\pi^2_{d_2}$ are equivalent if there exists a bijection between their sets of nodes that preserves the oriented edges and sends the decoration $d_1$ to a decoration equivalent to $d_2$. We call decorated forests the equivalence classes of such graphs, where we also add the empty forest $\textbf{1}$.

Analogously, we call exotic forests the equivalence classes of decorated forests $\pi_d$ such that for all $n>0$, $\abs{d^{-1}(n)}\in \{0,2\}$. A pair of nodes with matching non-zero decoration is called a liana. We gather the decorated and exotic forests in the sets $DF$ and $EF$ and write $\DD\FF=\Span_\R(DF)$, $\EE\FF=\Span_\R(EF)$.

The order of a decorated forest is
\[\abs{\pi_d}=\abs{d^{-1}(0)}+\frac{1}{2}\sum_{n>0} \abs{d^{-1}(n)}.\]
The number of automorphisms of a given decorated forest $\pi_d$ is the symmetry coefficient $\sigma(\pi_d)$.
\end{definition}

\begin{ex}
In \cref{ex_decorations}, all forests are exotic, except $\pi_{d_2}$ as the decoration $1$ appears more than two times.
The decorated forests of \cref{ex_decorations} have the order
\[
\abs{\pi_{d_1}}=3,\quad \abs{\pi_{d_2}}=\abs{\pi_{d_3}}=\abs{\pi_{d_4}}=\abs{\pi_{d_5}}=\abs{\pi_{d_6}}=2,
\]
and the symmetry coefficients satisfy
\[
\sigma(\pi_{d_1})=\sigma(\pi_{d_4})=\sigma(\pi_{d_5})=\sigma(\pi_{d_6})=1,\quad \sigma(\pi_{d_2})=\sigma(\pi_{d_3})=2.
\]
The list of all exotic and decorated forests of order one and two is presented in \cref{table:exotic_order_conditions,table:Isserlis_order_conditions}.
\end{ex}

\begin{remark}
Note that two decorated forests can be equivalent, while having non-equivalent decorations. For instance in \cref{ex_decorations}, the forests $\pi_{d_4}$, $\pi_{d_5}$, and $\pi_{d_6}$ are equivalent, but only the decorations $d_4$ and $d_5$ are equivalent.
Given two decorated forests $\pi_{d_1}$, $\pi_{d_2}$ with $d_2\leq d_1$, we write $m(\pi_{d_2},\pi_{d_1})$ the number of different finer decorations of $\pi_{d_1}$ that yield a decorated forest equivalent to $\pi_{d_2}$.
\end{remark}

Let us now equip $\EE\FF$ with algebraic structures.
On one hand, the concatenation product of two decorated forests, denoted $\pi_{d_1}\cdot\pi_{d_2}$, yields a decorated forest given by the union of the two graphs, with the decoration that preserves the nodes of same decoration in $\pi_{d_1}$ and $\pi_{d_2}$, but does not use the same non-zero decoration for nodes in $\pi_{d_1}\cdot\pi_{d_2}$:
\[
\forestOC\cdot\forestPC=\forestQC.
\]
On the other hand, the Grossman-Larson product $\pi_{d_1}^1 \diamond \pi_{d_2}^2$ concatenates and grafts the roots of $\pi_{d_1}^1$ on all nodes of $\pi_{d_2}^2$ in all possible ways (counting multiplicity), using different non-zero decorations for the nodes in $d_1$ and $d_2$.
For instance, we find
\[\forestRC\diamond \forestSC= \forestTC +\forestUC, \quad
\forestVC\diamond \forestWC=\forestXC+2\forestYC+\forestAD, \quad
\forestBD\diamond \forestCD=2\forestDD+2\forestED+4\forestFD+\forestGD
.\]
Let the deshuffle coproduct $\Delta\colon \EE\FF\rightarrow\EE\FF\otimes \EE\FF$:
\[
\Delta \pi=\sum_{\pi_1,\pi_2\in EF,\pi_1\cdot\pi_2=\pi} \pi_1\otimes\pi_2,\quad \Delta\textbf{1}=\textbf{1}\otimes \textbf{1}.
\]
For instance, we find (note that lianas cannot be split on different sides of the tensor product)
\[
\Delta \forestHD
=\textbf{1}\otimes \forestID
+\forestJD\otimes \forestKD
+\forestLD\otimes \forestMD
+\forestND\otimes \textbf{1}.
\]
The exotic forests are naturally equipped with two structures of Hopf algebras.
We refer to \cite{Bronasco22ebs,Bronasco22cef} for similar structures in the simpler context of SDEs with additive noise.
\begin{proposition}
The exotic forests equipped with the concatenation product and the Grossman-Larson product have a structure of Hopf algebras $(\EE\FF,\cdot,\Delta)$ and $(\EE\FF,\diamond,\Delta)$, graded by the order.
\end{proposition}

\begin{remark}
An exotic forest $\pi_d$ is called primitive if
\[\Delta\pi_d=\textbf{1}\otimes \pi_d+\pi_d\otimes \textbf{1}.\]
In contrast to the deterministic setting, primitive forests are not the exotic trees.
For instance, $\forestOD$ and $\forestPD$ are primitive as they cannot be written as a concatenation of exotic trees.
In a variety of numerical contexts, the primitive elements exactly correspond to the order conditions of the numerical methods, so that the number of order conditions is given by the number of primitive elements \cite{Ebrahimi24aso, Bronasco22cef}. One of the major difficulties of stochastic numerics for general SDEs is that the coefficient maps of stochastic Runge-Kutta methods are not characters w.r.t.\ts concatenation in general, so that one has to consider all decorated forests for the order conditions. Hence the high number of order conditions in \cref{theorem:RK_conditions_general}.
\end{remark}

\subsection{Algebraic expansion of flows with decorated and exotic forests}

The decorated and exotic forests represent elementary differentials via the use of the elementary differential map $F^{\text{dec}}$.
\begin{definition}
Let a decorated forest $\pi_d=(V,E,d)$.
Then, the decorated elementary differential map $F^{\text{dec}}$ is the differential operator acting on test functions $\phi \in \CC_P^\infty(\R^d)$ satisfying $F^{\text{dec}}(\textbf{1})[\phi](x)=\phi(x)$ and
\[
F^{\text{dec}}(\pi_d)[\phi](x)=\sum_{\underset{w\in V}{i_w=1,\dots,d}} \sum_{\underset{p_0=0,\ p_{n_1}\neq p_{n_2} \text{ if } n_1\neq n_2}{p_n=1,\dots,m,\ n\in \Img(d)\setminus 0}} \phi_{I_{R}}(x) \prod_{v\in V} f^{p_{d(v)},i_v}_{I_{\Pi(v)}}(x),
\]
where $R\subset V$ is the set of roots, $\Pi(v)$ denotes the set of predecessors of $v$, and $I_{S}=i_{w_1}\dots i_{w_n}$ for $S=\{w_1,\dots,w_n\}$.
Analogously, the exotic elementary differential map on $\pi_d\in EF$ satisfies $F^{\text{exo}}(\textbf{1})[\phi](x)=\phi(x)$ and
\[
F^{\text{exo}}(\pi_d)[\phi](x)=\sum_{\underset{w\in V}{i_w=1,\dots,d}} \sum_{\underset{p_0=0}{p_n=1,\dots,m,\ n\in \Img(d)\setminus 0}} \phi_{I_{R}}(x) \prod_{v\in V} f^{p_{d(v)},i_v}_{I_{\Pi(v)}}(x).
\]
\end{definition}

The two elementary differential maps $F^{\text{dec}}$ and $F^{\text{exo}}$ are very similar, the only difference lying in the additional summing restriction $p_{n_1}\neq p_{n_2}$ in the definition of $F^{\text{dec}}$.
For instance, the forest $\pi_{d_3}$ from \cref{ex_decorations} yields
\begin{align*}
F^{\text{dec}}(\pi_{d_3})[\phi]&=\sum_{i,j,k,l=1}^d \sum_{\underset{p\neq q}{p,q=1}}^m \phi_{ijk} f^{p,i}_{l} f^{p,l} f^{q,j} f^{q,k},\\
F^{\text{exo}}(\pi_{d_3})[\phi]&=\sum_{i,j,k,l=1}^d \sum_{p,q=1}^m \phi_{ijk} f^{p,i}_{l} f^{p,l} f^{q,j} f^{q,k}.
\end{align*}
We observe in particular
\[F^{\text{exo}}(\pi_{d_3})=F^{\text{dec}}(\pi_{d_3})+F^{\text{dec}}(\pi_{d_2}).\]
We refer to \cref{table:exotic_order_conditions,table:Isserlis_order_conditions} for further examples.
The exotic formalism is simpler as it does not require to work with different cases depending on the equal values in $p_n$, which quickly becomes tedious as the order of the method goes up.

The following result, in the spirit of \cite{Rota64otf}, allows one to translate elementary differentials between the exotic and decorated formalisms.
\begin{proposition}
\label{prop:Moebius}
Let $\pi_d\in EF$ be an exotic forest, then the exotic differential map satisfies
\[
F^{\text{exo}}(\pi_d)=\sum_{\underset{d\leq d_0}{\pi_{d_0}\in DF}} F^{\text{dec}}(\pi_{d_0}),
\]
and $F^{\text{exo}}$ can thus be extended to $\DD\FF$ with this identity.
Then, the Moebius inversion formula holds for $\pi_d\in DF$,
\[
F^{\text{dec}}(\pi_{d})=\sum_{\underset{d\leq d_0}{\pi_{d_0}\in DF}} \mu(\pi_{d},\pi_{d_0}) F^{\text{exo}}(\pi_{d_0}),
\]
where the Moebius function is
\[
\mu(\pi_{d_1},\pi_{d_1})=1,\quad \mu(\pi_{d_1},\pi_{d_2})=-\sum_{\underset{d_1\leq d<d_2}{\pi_{d}\in DF}} \mu(\pi_{d_1},\pi_d).
\]
\end{proposition}

The elementary differential translates the operations on forests to operations on differential operators.
\begin{lemma}
\label{lemma:comp_GL}
The composition of differential operators is given by the Grossman-Larson product: for $\pi_{d}$, $\hat{\pi}_{\hat{d}}\in EF$, one has for all $\phi\in\CC_P\infty(\R^d)$,
\[
(F^{\text{exo}}(\pi_d)\circ F^{\text{exo}}(\hat{\pi}_{\hat{d}}))[\phi]=F^{\text{exo}}(\pi_d)[F^{\text{exo}}(\hat{\pi}_{\hat{d}})[\phi]]=F^{\text{exo}}(\pi_{d}\diamond\hat{\pi}_{\hat{d}})[\phi].
\]
\end{lemma}

The generators of the SDEs \eqref{equation:SDE_Ito}/\eqref{equation:SDE_Strato} are represented by exotic forests as
\[
\LL_{\text{Itô}}=F^{\text{exo}}(\forestQD+\frac{1}{2}\forestRD), \quad \LL_{\text{Strato}}=F^{\text{exo}}(\forestSD+\frac{1}{2}\forestTD+\frac{1}{2}\forestUD).
\]
The expansion \eqref{equation:exact_expansion} rewrites in terms of exotic forests, using \cref{lemma:comp_GL}.
\begin{proposition}
\label{prop:expansion_GL}
The Taylor expansion of the flow \eqref{equation:exact_expansion} is represented by the Grossman-Larson exponential $\exp^\diamond(\tau)=\sum_{n=0}^\infty \frac{1}{n!} \tau^{\diamond n}$, that is,
\begin{align*}
u_{\text{Itô}}(x,h)&=F^{\text{exo}}(\exp^\diamond(h\forestVD+\frac{h}{2}\forestWD))[\phi],\\
u_{\text{Strato}}(x,h)&=F^{\text{exo}}(\exp^\diamond(h\forestXD+\frac{h}{2}\forestYD+\frac{h}{2}\forestAE))[\phi].
\end{align*}
\end{proposition}

\subsection{Decorated and exotic S-series for stochastic numerics}

The main tool for representing the Taylor expansions of our analysis are S-series, which generalise Butcher series \cite{Butcher69teo, Butcher72aat, Hairer06gni} to represent differential operators.
\begin{definition}
Given 1-forms, called coefficient maps, $a\in \DD\FF^*$ and $b\in \EE\FF^*$, the associated decorated and exotic S-series are the following formal power series
\begin{align*}
S_h^{\text{dec}}(a)[\phi]&=\sum_{\pi_d\in DF} h^{\abs{\pi_d}} \frac{a(\pi_d)}{\sigma(\pi_d)} F^{\text{dec}}(\pi_d)[\phi],\\
S_h^{\text{exo}}(b)[\phi]&=\sum_{\pi_d\in EF} h^{\abs{\pi_d}} \frac{b(\pi_d)}{\sigma(\pi_d)} F^{\text{exo}}(\pi_d)[\phi].
\end{align*}
We denote $S^{\text{dec}}(a)=S_1^{\text{dec}}(a)$ and $S^{\text{exo}}(a)=S_1^{\text{exo}}(a)$.
\end{definition}

\begin{remark}
Similar to the exotic S-series being formal sums indexed by exotic forests, the exotic B-series are formal series indexed by exotic trees, whose use is central in the numerical approximation for the invariant measure of ergodic stochastic dynamics \cite{Laurent20eab, Laurent21ocf, Bronasco22ebs, Bronasco22cef}.
It is proven in \cite{Laurent23tue} that the exotic B-series from the additive noise case are not just a combinatorial tool for the simplification of tedious calculations, but are a universal geometric object characterised by the properties of locality and orthogonal equivariance (see also \cite{McLachlan16bsm, MuntheKaas16abs}).
To the best of our knowledge, such geometric properties have not yet been studied for the exotic and decorated S-series presented in this work.
\end{remark}

The Taylor expansion \eqref{equation:exact_expansion} of the flow and the corresponding expansion
\eqref{equation:num_expansion} of the stochastic Runge-Kutta method \eqref{equation:def_RKsto} can
typically be written in terms of decorated S-series, but unfortunately not by exotic S-series for the numerical flow in general, in opposition to the additive noise case. The previous works rely on a set of forests that is unnecessarily bigger than $DF$. The use of decorated forests allows us to avoid repeating the same condition multiple times and to take into account the symmetries of the conditions: one decorated forest stands for one order condition exactly.
\begin{theorem}
\label{thm:flow_exp_S_series}
The exact flow of \eqref{equation:SDE_Ito}/\eqref{equation:SDE_Strato} is given by an exotic S-series:
\[
\E[\phi(X(h))|X_0=x]=S_h^{\text{exo}}(e)[\phi](x).
\]
Consider a stochastic Runge-Kutta method of the form \eqref{equation:def_RKsto} whose coefficients satisfy \cref{ass:RV_conditions}. Then, the expansion \eqref{equation:num_expansion} of the integrator is well defined and is given by a decorated S-series:
\[\E[\phi(X_1)|X_0=x]=S_h^{\text{dec}}(a)[\phi](x).\]
\end{theorem}

The coefficient $e$ can be computed by iteration of the Grossman-Larson product as in \cref{prop:expansion_GL} and then multiplying by the symmetry coefficient (see also the expression in \cite{Rossler04ste}).
For instance, one finds in the Itô case \eqref{equation:SDE_Ito}:
\[
\exp^\diamond(h\forestBE+\frac{h}{2}\forestCE)
=h\Big(\forestDE+\frac{1}{2}\forestEE\Big)
+h^2\Big(\frac{1}{2}\forestFE+\frac{1}{2}\forestGE+\dots\Big)+\dots,
\]
so that $e$ satisfies
\[
\frac{e}{\sigma}(\forestHE)=1,\quad
\frac{e}{\sigma}(\forestIE)=\frac{1}{2},\quad
\frac{e}{\sigma}(\forestJE)=0,\quad
\frac{e}{\sigma}(\forestKE)=\frac{1}{2},\quad
\frac{e}{\sigma}(\forestLE)=\frac{1}{2},\dots
\]
Alternatively, we provide in \cref{prop:BCK} and equation \eqref{eq:BCK_coef_e} an explicit expression of $e$ using an extension of the Butcher-Connes-Kreimer Hopf algebra.
The values of $e$ for exotic forests of order up to two are presented in \cref{table:exotic_order_conditions}.

The coefficient map of stochastic Runge-Kutta methods \eqref{equation:def_RKsto} is deduced from the decorated graph, similarly to the deterministic literature \cite[Chap.\ts III]{Hairer06gni}. One arbitrarily labels the nodes of the forest and writes $z_r^{p_d}$ if a root is labeled with $r$ and decorated with the decoration $d$, with the convention $p_0=0$. Similarly, we write $Z_{v,w}^{p,q}$ if an edge goes from the node $w$ with decoration $q$ to the node $v$ with decoration $p$. The coefficient $a$ is then obtained by summing on every indices involved and taking the expectation.
We provide an example and refer the reader to \cite[Prop.\ts 4.3]{Bronasco22ebs} and \cite{Rossler06rta,Debrabant08bsa} for further details,
\begin{align*}
\pi_d&=\forestME,\\
a(\pi_d)&=\sum_{a,b,c,\dots=1}^s \sum_{p_1,p_2,p_3=1}^m \E[z_a^0 Z_{a,b}^{0,p_1} Z_{b,c}^{p_1,p_2} Z_{a,d}^{0,p_3} Z_{a,e}^{0,0} z_f^{p_2} Z_{f,g}^{p_2,0} z_h^{p_3} Z_{h,i}^{p_3,p_1}].
\end{align*}

\begin{remark}
A 1-form $a\in \HH^*$ over an algebra $(\HH,\cdot)$ is called a character if
\[
a(x\cdot y)=a(x)a(y),\quad x,y\in \HH.
\]
The coefficient map $e\in \EE\FF^*$ of the exact flow naturally is a character. In most numerical contexts, the methods are represented by characters on some Hopf algebra. A major difficulty of stochastic numerics with multiplicative noise is that the coefficient map $a$ of stochastic Runge-Kutta methods is not a character (as $\E[XY]\neq\E[X]\E[Y]$ in general).
\end{remark}

For the approximation of SDEs with additive noise, it is sufficient to use only Gaussian random Runge-Kutta coefficients. In this context, the Isserlis theorem \cite{Isserlis18oaf, MuntheKaas25asp} (also called the Wick formula) guarantees that the S-series of the numerical solution rewrites as an exotic series. Gaussian Runge-Kutta coefficients do not suffice for high order in the case of multiplicative noise. We thus rely on decorated forests and we split in \cref{theorem:RK_conditions_general} the order conditions into two sets, namely exotic and Isserlis order conditions. The Isserlis order conditions allow to write the Taylor expansion \eqref{equation:num_expansion} as an exotic S-series up to the order studied and are computed using the following result.
\begin{theorem}
\label{theorem:Isserlis_discrete}
Consider a numerical method whose expansion takes the form of a decorated S-series with coefficient map $a$. Then, the expansion writes as an exotic S-series if for all $\pi_d\in DF$, the following condition, that we call Isserlis condition, is satisfied,
\[
a(\pi_d)=\sum_{\underset{d_0\leq d}{\pi_{d_0}\in EF}} \frac{\sigma(\pi_d)}{\sigma(\pi_{d_0})} a(\pi_{d_0})
=\sum_{\underset{\pi_{d_0}\in EF}{d_0\leq d}} a(\pi_{d_0}).
\]
\end{theorem}

\Cref{theorem:Isserlis_discrete} is a straightforward consequence of the identity $m(\pi_{d_0},\pi_d)=\frac{\sigma(\pi_d)}{\sigma(\pi_{d_0})}$ and \cref{prop:Moebius}.
We emphasize that the first sum in \cref{theorem:Isserlis_discrete} is indexed by all the exotic forests with a finer decoration than $d$, while the second one is indexed by all the exotic decorations finer than $d$, so that it allows some repetition.
For example, we have the following Isserlis condition of order two:
\[a(\forestNE)=a(\forestOE)+a(\forestPE)+a(\forestQE)=2a(\forestRE)+a(\forestSE).\]

From the expansions in exotic and decorated S-series of the exact and numerical flow and the characterisation of Isserlis conditions, we derive the weak order conditions of \cref{theorem:RK_conditions_general}.
\begin{proof}[Proof of \cref{theorem:RK_conditions_general}]
Thanks to \cref{thm:flow_exp_S_series}, the exact and numerical flow write respectively as an exotic S-series $S_h^{\text{exo}}(e)$ and a decorated S-series $S_h^{\text{dec}}(a)$.
The Isserlis order conditions from \cref{theorem:Isserlis_discrete} impose that the numerical flow writes as the exotic S-series $S_h^{\text{exo}}(a)$. Then, the order conditions are obtained by identifying the coefficient maps on exotic forests: $a(\pi)=e(\pi)$ for $\pi\in EF$. Together with \cref{ass:RV_conditions}, they ensure that the prerequisites for \cref{prop:Convergence} are satisfied.
\end{proof}

To complete our description of exotic S-series, let us describe the composition of differential operators represented by exotic S-series using the celebrated Butcher-Connes-Kreimer (BCK) Hopf algebra \cite{Connes98har, Kreimer99lfq}.
\begin{definition}
\label{def:adm_cut_BCK}
A cut $c$ of $\pi\in EF$ is a (possibly empty) choice of edges of $\pi$. Removing these edges yields a forest $W_c(\pi)$, the component of $W_c(\pi)$ containing the roots of $\pi$ is written $R_c(\pi)$ and the other components are gathered in the forest $P_c(\pi)$.
A cut is admissible, written $c\in\Adm(\pi)$, if any path from a root to a leaf of $\pi$ has at most one cut and if $R_c(\pi)$, $P_c(\pi)\in \EE\FF$.
\end{definition}

We define the exotic extension of the BCK Hopf algebra in the following result. The proof is analogous to the additive noise case \cite{Bronasco22ebs, Bronasco22cef} and is thus omitted (see also \cite{Panaite00rtc, Hoffman03cor}).
\begin{proposition}
Define the BCK coproduct on $\EE\FF$ by
\[
\Delta_{BCK}\textbf{1}=\textbf{1}\otimes \textbf{1}, \quad \Delta_{BCK}\pi=\pi\otimes \textbf{1}+\sum_{c\in\Adm(\pi)} P_c(\pi)\otimes R_c(\pi).
\]
Then, $(\EE\FF,\cdot,\Delta_{BCK})$ is a graded connected commutative Hopf algebra and its graded dual is isomorphic (up to the symmetry coefficient) to the Grossman-Larson Hopf algebra $(\EE\FF,\diamond,\Delta)$.
\end{proposition}

The main difference with the deterministic setting \cite{Chartier10aso} and the standard Hopf algebra formalisms over decorated forests \cite{Foissy21aso} is that the nodes sharing the same decoration cannot be split by the Butcher-Connes-Kreimer coproduct:
\begin{align*}
    \Delta_{BCK} \forestTE
    &= \mathbf{1} \otimes \forestUE
    +\forestVE \otimes \forestWE
    +\forestXE \otimes \forestYE
    +\forestAF \otimes \forestBF
    + \forestCF \otimes \mathbf{1}.
\end{align*}

The composition of exotic S-series is described by the BCK coproduct (see also \cite{Debrabant11cos} for an analogous description of the composition without the Hopf algebra formalism).
\begin{proposition}
\label{prop:BCK}
Let $a,b\in\EE\FF^*$, then the composition of exotic S-series satisfies
\[
S^{\text{exo}}(a)\circ S^{\text{exo}}(b)=S^{\text{exo}}(a* b),\quad a* b=\mu\circ (a\otimes b)\circ \Delta_{BCK},
\]
where $\mu$ is the multiplication and $*$ is called the composition law.
\end{proposition}

Define the coefficient map of the Itô generator $l\colon \EE\FF \rightarrow\R$ by $l=\delta_{\forestDF}+\delta_{\forestEF}$ (respectively $l=\delta_{\forestFF}+\delta_{\forestGF}+\frac{1}{2}\delta_{\forestHF}$ for Stratonovich), where $\delta_{\tau}(\hat{\tau})=\ind_{\tau=\hat{\tau}}$. The associated S-series yields the generator $S^{\text{exo}}(l)=\LL$.
Then, the coefficient map $e$ of the exact flow satisfies
\begin{equation}
\label{eq:BCK_coef_e}
u(x,h)=S^{\text{exo}}(e)[\phi](x),\quad e=\exp^*(hl).
\end{equation}

\section{Conclusion}
\label{section:conclusion}

In this paper, we provided a novel approach based on specific choices of random Runge-Kutta coefficients that allows to greatly reduce the number of order conditions for stochastic Runge-Kutta methods of high weak order in the case of multiplicative noise.
We provide a collection of new methods with similar accuracy and cheaper cost compared to the literature. The methods are optimal in the sense that they use the minimal number of function evaluations.
The analysis of the order conditions relies on exotic and decorated S-series and the associated Hopf algebra structures. The algebraic approach allows to obtain a formalism where one forest exactly corresponds to one order condition, to emphasize the central role of exotic series in stochastic numerics, and to identify the major algebraic difficulties brought by multiplicative noise.

The present work raises several new questions, from both algebraic and numerical perspectives, in the design of efficient discretisations. First, it would greatly simplify the analysis to find a class of methods whose Taylor expansions write as exotic series directly.
Moreover, the design of the random variables was done by hand, and a formalisation of the definition of random variables satisfying given moment identities would help in the challenging calculations for higher order.
The simplifying approach and the new methods presented in the present paper could be extended to other contexts such as, for instance, the creation of invariant measure-preserving methods for ergodic dynamics \cite{Bronasco25els}, symplectic schemes for the study of stochastic Hamiltonian systems \cite{Hong22sio}, or for stochastic integration on manifolds.
For the latter, the exotic Butcher series formalism is extended to the study of stochastic Lie-group methods with additive noise in \cite{Bronasco25hoi}. The extension of such methods for multiplicative noise, for sampling ergodic dynamics, and the study of the associated algebraic structures (in the spirit of \cite{Busnot25pha, Busnot26tft}) is exciting matter for future work.

\bigskip

\noindent \textbf{Acknowledgements.}\
A.\ts Busnot Laurent acknowledges the support from the programs ANR-25-CE40-2862-01 (MaStoC - Manifolds and Stochastic Computations) and ANR-11-LABX-0020 (Labex Lebesgue) and would like to thank E.\ts Bronasco and P.\ts Catoire for insightful discussions.

\bibliography{ma_bibliographie}

\vskip-1ex
\begin{appendices}

\section{Implicit-explicit stochastic Runge-Kutta methods}
\label{section:IMEX}

We present in this section a collection of new IMEX methods with optimal number of stages. Such methods could typically be used for solving multiscale stochastic systems.

\paragraph*{Itô diagonally implicit-explicit ($c=\frac14$)}\mbox{}\par
\[
  \begin{array}{c|cc}
    \frac12 & \frac12 & 0 \\ \hline
    0 & 0 & 0 \\
    \frac12 & \frac{1}{2} & 0 \\ \hline
    1 & 0 & 1
  \end{array}
\]

\paragraph*{Itô explicit-diagonally implicit ($c=\frac12$)}\mbox{}\par
\[
\begin{array}{cc|cc}
    0 & 0 & 0 & 0 \\
    1 & 0 & 1 & 0 \\ \hline
    0 & 0 & 1 & 0 \\
    \frac12 & 0 & -\frac12 & 1 \\ \hline
    \frac12 & \frac12 & 0 & 1
  \end{array}
\]

\paragraph*{Stratonovich diagonally implicit-explicit ($c=\frac14$)}\mbox{}\par
\[
  \begin{array}{c|cccc|cccc}
    \frac12 & 0 & \frac12 & 0 & 0 \\
    \hline
    0 & 0 & 0 & 0 & 0 & 0 & 0 & 0 & 0 \\
    0 & \frac12 & 0 & 0 & 0 & \frac12 & 0 & 0 & 0 \\
    \frac32 & -1 & \frac32 & 0 & 0 & -\frac12 & \frac12 & 0 & 0 \\
    \frac32 & -\frac12 & \frac32 & -\frac12 & 0 & -\frac32 & \frac32 & 1 & 0 \\ \hline
    1 & 0 & \frac23 & \frac16 & \frac16
  \end{array}
\]

\paragraph*{Stratonovich explicit-diagonally implicit ($c=\frac12$)}\mbox{}\par
\[
  \begin{array}{cc|ccc|ccc}
    0 & 0 & 0 & 0 & 0 \\
    1 & 0 & 1 & 0 & 0 \\
    \hline
    0 & 0 & 0 & 0 & 0 & \frac12 & 0 & 0 \\
    \frac12 & 0 & \frac12 & 0 & 0 & \frac{3-2\sqrt{3}}{6} & \frac{\sqrt{3}}{6} & 0 \\
    \frac12 & 0 & \frac12 & 0 & 0 & \frac{-3+2\sqrt{3}}{6} & \frac12 & \frac{3-\sqrt{3}}{6} \\ \hline
    \frac12 & \frac12 & 0 & \frac12 & \frac12
  \end{array}
\]

\paragraph*{Stratonovich diagonally implicit ($c=\frac14$)}\mbox{}\par
\[
  \begin{array}{c|ccc|ccc}
    \frac12 & \frac12 & 0 & 0 \\
    \hline
    0 & 0 & 0 & 0 & \frac12 & 0 & 0 \\
    \frac12 & \frac12 & 0 & 0 & \frac{3-2\sqrt{3}}{6} & \frac{\sqrt{3}}{6} & 0 \\
    \frac12 & \frac12 & 0 & 0 & \frac{-3+2\sqrt{3}}{6} & \frac12 & \frac{3-\sqrt{3}}{6} \\ \hline
    1 & 0 & \frac12 & \frac12
  \end{array}
\]

\end{appendices}

\end{document}